\documentclass[12pt]{amsart}
\usepackage{amsmath,amssymb,hyperref}
\usepackage[all,cmtip]{xy}
\usepackage[english,french]{babel}

\title{Rationality of complete intersections of two quadrics}
\author[Hassett and Tschinkel]{Brendan Hassett and Yuri Tschinkel,
with an appendix by Jean-Louis Colliot-Th\'el\`ene}
\date{April 19, 2019}

\theoremstyle{plain}
\newtheorem{cons}{Construction}

\theoremstyle{plain}
\newtheorem{prop}{Proposition}

\newtheorem{theo}[prop]{Theorem}
\newtheorem{coro}[prop]{Corollary}

\newtheorem{lemm}[prop]{Lemma}

\newtheorem{CTtheo}{Th\'{e}or\`{e}me}[section]
\newtheorem{CTlem}[CTtheo]{Lemme}
\newtheorem{CTprop}[CTtheo]{Proposition}

\theoremstyle{definition}

\newtheorem{ques}[prop]{Question}

\newtheorem{rema}[prop]{Remark}

\newtheorem{exam}[prop]{Example}

\newcommand{\C}{\mathbb C}
\newcommand{\F}{\mathbb F}
\renewcommand{\P}{\mathbb P}
\newcommand{\Q}{\mathbb Q}
\newcommand{\R}{\mathbb R}
\newcommand{\Z}{\mathbb Z}

\newcommand{\bA}{\mathbb A}
\newcommand{\bC}{\mathbb C}
\newcommand{\bF}{\mathbb F}
\newcommand{\bG}{\mathbb G}
\newcommand{\bN}{\mathbb N}
\newcommand{\bP}{\mathbb P}

\newcommand{\bR}{\mathbb R}
\newcommand{\bZ}{\mathbb Z}

\newcommand{\cC}{\mathcal C}
\newcommand{\cE}{\mathcal E}

\newcommand{\cI}{\mathcal I}
\newcommand{\cL}{\mathcal L}
\newcommand{\cM}{\mathcal M}
\newcommand{\cO}{\mathcal O}

\newcommand{\cX}{\mathcal X}
\newcommand{\cZ}{\mathcal Z}

\newcommand{\fS}{\mathfrak S}

\newcommand{\wX}{\widetilde X}
\newcommand{\hX}{\widehat  X}

\newcommand{\ra}{\rightarrow}

\newcommand{\Alb}{\operatorname{Alb}}
\newcommand{\Aut}{\operatorname{Aut}}
\newcommand{\Bl}{\operatorname{Bl}}
\newcommand{\Br}{\operatorname{Br}}
\newcommand{\CH}{\operatorname{CH}}

\newcommand{\Gal}{\operatorname{Gal}}

\newcommand{\Gr}{\operatorname{Gr}}

\newcommand{\IntJ}{\operatorname{IJ}}
\newcommand{\Jac}{\operatorname{J}}
\newcommand{\Nm}{\operatorname{Nm}}
\newcommand{\OGr}{\operatorname{OGr}}
\newcommand{\Pic}{\operatorname{Pic}}
\newcommand{\Sec}{\operatorname{Sec}}
\newcommand{\Sym}{\operatorname{Sym}}
\newcommand{\Tr}{\operatorname{Tr}}

\begin{document}

\selectlanguage{english}

\begin{abstract}
We study rationality problems for smooth complete
intersections of two quadrics. 
We focus on the three-dimensional case, 
with a view toward understanding the invariants governing
the rationality of a geometrically rational threefold over a nonclosed field.
\end{abstract} 

\maketitle

\section{Introduction}
\label{sect:intro}

Rationality problems for geometrically rational surfaces $X$ over nonclosed fields $k$ have been intensely studied and are well understood. 
Their minimal models over $k$ are either conic bundles or del Pezzo surfaces.  
In addition to the existence of $k$-rational points $X(k)$, the essential information is encoded in the Picard group $\Pic(\bar{X})$, viewed as 
a Galois-module. Rationality can be detected by analyzing the Galois orbit structure on the set of exceptional curves -- with a few exceptions, minimal surfaces $X$ cannot be rational. This allows to completely settle the question of rationality of $X$. Stable rationality remains a challenging open problem. 

In dimension 3, the minimal model program leads to conic bundles, del Pezzo fibrations, and Fano varieties. There is an extensive literature
on rationality of these varieties over algebraically closed fields, see, e.g., \cite{H-I}.
Recently, there has been decisive progress on the 
stable rationality problem, essentially settling this problem in dimension 3, with the exception of varieties birational to a cubic 
\cite{HKT-conic}, \cite{HT-fano}, \cite{krylov-okada}.
These results relied on the specialization method introduced by Voisin \cite{voisin} and developed in \cite{ct-pirutka}, 
see also \cite{Voisin-survey},  \cite{ct-surve}. 
However, little is known about rationality properties of geometrically rational threefolds over nonclosed fields such as  
finite fields or function fields of curves.

Here, we focus on the simplest geometrically rational, nontoric, example, the 
complete intersection of two quadrics 
\begin{equation}
\label{eqn:x}
X\subset \bP^5,
\end{equation}
with a view toward understanding the invariants governing
the rationality of a geometrically rational threefold over relatively
simple ground fields. 
Indeed, rationality properties of toric varieties are controlled by the Picard group, as a Galois module, similar to the case of surfaces; 
the classification of nonrational tori can be found in \cite{kun}. 
For $X$ a smooth intersection of two quadrics, we have 
$$\Pic(\bar{X})=\bZ,$$
with trivial Galois action, and the Brauer group $\Br(X)$ is  also trivial. 
As we will see, the intermediate Jacobian $\IntJ(X)$ is the Jacobian of a curve of genus $2$, over $k$.   
Similarly, there are no obstructions from the
birational rigidity viewpoint. Are there any obstructions to rationality?   
 
We find them in specializations. 
One of our main results, Theorem~\ref{theo:main}, is a proof of 
failure of stable rationality of general $X$ as in \eqref{eqn:x}.  
We use two different specializations of $X$. In Section~\ref{sect:toric} we use nonrational toric threefolds over $k$ to obtain smooth examples over $k((\tau))$ that are nonrational but admit rational points. In Section~\ref{sect:irr}, we
work over $k=\bC(t)$ and view $X$ as a fourfold over $\bC$, admitting a quadric surface bundle over $\bP^1\times \bP^1$. The failure of stable rationality of such $X$ can be proven using the techniques of \cite{HPT16};
this implies that $X$ is not stably rational over $k$.  
 
On the other hand, natural rationality constructions interact in
unexpected ways.
In Section~\ref{sect:background}, we recall the geometry of the varieties of projective subspaces on higher-dimensional intersections of two quadrics. The most classical rationality construction -- projection from lines -- is discussed in Section~\ref{sect:ratbylines}. Another useful -- and {\em a priori} distinct -- rationality construction is presented in Section~\ref{sect:FQS}. Interestingly, this forces 
the existence of a line over the ground field (Theorem~\ref{theo:line}). 
In Section~\ref{sect:cocycle} we study the connection to the intermediate Jacobian of $X$ and its twisted forms; Question~\ref{ques:CT}(ii)
 is a natural outgrowth of our examples.
Section~\ref{sect:odddegree} presents a related 
rationality result (Theorem~\ref{theo:secant}) asserting
that odd-degree curves force the existence of a line and rationality.
We then explore what happens as the corresponding cocycle associated
with the intermediate Jacobian collapses.
For example, $X$ might admit a pair of skew lines over the ground field, addressed in Section~\ref{sect:skew}. 
Given the importance of rational curves in all our constructions, we explore their geometry in Section~\ref{sect:higher}. Real pencils of quadrics have been
extensively studied; rationality in the three-dimensional case is discussed
in Section~\ref{sect:real}.

Throughout, we work over a basefield $k$ that has characteristic not equal to 
two.

\

\noindent
{\bf Acknowledgments:} 
The first author was partially supported by NSF grants 1551514 and 1701659, and the Simons Foundation;
the second author was partially supported by NSF grant 1601912.
We are grateful to J.-L.~Colliot-Th\'el\`ene
for conversations on this topic and helpful comments on our manuscript.
Olivier Wittenberg provided constructive suggestions on Section~\ref{sect:real}.

\section{Geometric background}
\label{sect:background}

In this section, $X \subset \bP^n$ is a smooth complete intersection
of two quadrics over $k$.

\subsection{Varieties of linear subspaces}

Let $F_r(X) \subset \bG(r,n)$ denote the variety of $r$-dimensional
linear subspaces in $X$. It has expected dimension
$$
e(r,n) = (r+1)(n-2r-2)
$$
which is non-negative provided $n\ge 2r+2$.  

\begin{prop}
If $n\ge 2r+2$ then $F_r(X)$ is smooth and nonempty of
dimension $e(r,n)$.
If $n>2r+2$ it is also connected.
\end{prop}
The nonsingularity result is \cite[Th.~2.6]{ReidThesis} and
nonemptyness is addressed in the introductions to Chapters 3 and 4
of \cite{ReidThesis}.
The connectedness assertion is \cite[Th.~2.1]{DebMan}.

The adjunction formula \cite[p.~555]{DebMan} gives the dualizing
sheaf
$$\omega_{F_r(X)} \simeq \cO_{F_r(X)}(2r-n+3),$$
which is always non-negative provided $e(r,n)>0$. 
It is negative unless $n=2g+1$ and $r=g-1$ for some $g\in \bN$.

\subsection{Connections with hyperelliptic curves and vector bundles}
Assume that $n=2g+1$ and consider the fibration 
$$\varrho:\Bl_X(\bP^{2g+1}) \ra \bP^1$$
associated with the pencil of quadrics and the relative variety of
maximal isotropic subspaces
$$F_g(\Bl_X(\bP^{2g+1})/\bP^1) \ra \bP^1$$
which factors
$$
F_g(\Bl_X(\bP^{2g+1})/\bP^1)
\stackrel{\varpi}{\ra} C \stackrel{\gamma}{\ra} \bP^1,
$$
where $\varpi$ is smooth and $\gamma$ is a double cover
branched over $D \subset \bP^1$ with $|D|=2g+2$.

Over $\bar{k}$, $X$ may be diagonalized
$$
X=\{a_{00}x_0^2+\cdots + a_{0 2g+1} x_{2g+1}^2= 
a_{10}x_0^2+\cdots + a_{1 2g+1} x_{2g+1}^2= 0 \},
$$
and thus admits automorphisms by diagonal matrices
with entries $\pm 1$
$$(\bZ/2\bZ)^{2g+1} \subset \Aut(\bar{X}).$$
As a Galois module, this may be represented as
$$H=\left<p_1,\ldots,p_{2g+2}\right> \subset \Pic(\bar{C})/\left<g^1_2\right>,$$
where the $p_i$ are the branch points. This has the relations
$$2p_i \equiv 0, p_1+\cdots+p_{2g+2} \equiv 0.$$

\

Desale-Ramanan \cite{DR} obtain:
\begin{prop} \label{prop:DR}
Assume $k$ is algebraically closed.
\begin{itemize}
\item 
$F_{g-1}(X)$ is a torsor over the 
Jacobian of $C$.
\item
$F_{g-2}(X)$ is the moduli space of rank-two vector
bundles $\cE$ with fixed odd determinant, i.e., an isomorphism
$$\det(\cE) \simeq \cL.$$
\end{itemize}
\end{prop}
The generic such bundle admits automorphisms by $\pm 1$ and an action
$$\cE \mapsto \cE \otimes \cM$$
where $\cM^{\otimes 2} \simeq \cO_C$.

\subsection{Quadric pencils} \label{subsect:QP}
Consider the fibration 
$$\varrho: P:=\Bl_X(\bP^n) \ra \bP^1$$
associated with the pencil of quadrics.
\begin{prop}[Brumer Theorem \cite{Brumer}]
The variety $X$ admits a rational point if and only if $\varrho$ admits a section.
\end{prop}

Here we assume $n=2g+1$ and write $\OGr_{g+1}$ for one
of the two isomorphic connected components
of the Grassmannian of maximal isotropic subspaces in a nondegenerate
quadratic form in $2g+2$ variables.
The relative variety of
planes
$$F_g(P/\bP^1) \ra \bP^1$$
factors
$$
F_g(P/\bP^1) \stackrel{\varpi}{\ra} C \stackrel{\gamma}{\ra} \bP^1,
$$
where $\varpi$ is an \'etale $\OGr_g$-bundle and $\gamma$ is a double cover.
The standard theory of quadratic forms yields a class $\alpha \in \Br(C)[2]$.

Now assume $n=5$ and $g=2$.
Here $\OGr_3$ is a Brauer-Severi threefold, geometrically
isomorphic to $\bP^3$. Thus 
$$F_2(P/\bP^1) \stackrel{\varpi}{\ra} C $$
is an \'etale $\bP^3$-bundle, hence the index of $\alpha$ divides four.
Intersection with $X$ gives an isomorphism
$$F_2(P/\bP^1) \stackrel{\sim}{\ra} R_2(X)$$
to the variety $R_2(X)$ of conics on $X$. 

\begin{rema}
It follows that if $X$ contains a conic over $k$ then $C(k) \neq \emptyset$.
The converse holds if  $\alpha=0$ or
$k$ is a $C_1$-field.
\end{rema}

\section{Rationality via lines}
\label{sect:ratbylines}
We recall a standard rationality construction:
Let $X\subset \bP^n$ be a smooth complete intersection of two quadrics containing a line $\ell$. 
Projecting from $\ell$ induces a birational morphism
$$\beta:\Bl_{\ell}(X) \ra \bP^{n-2}$$
that blows down all the lines in $X$ incident to $\ell$. Fix coordinates $x_0,\ldots,x_n$ so that
$\ell=\{x_2=\ldots=x_n=0\}$ and 
$$X=\{L_{00}x_0+L_{01}x_1+Q_0 = L_{10}x_0+L_{11}x_1+ Q_1=0\},
$$
where the $L_{ij}$ are linear and the $Q_i$ quadratic in $x_2,\ldots,x_n$.
The inverse mapping $\beta^{-1}$ is given by a 
linear series of cubics associated with the $2\times 2$ minors of a matrix
$$\left( \begin{matrix} L_{00} & L_{01} \\
				  L_{10} & L_{11} \\
				  Q_0 & Q_1 
	\end{matrix} \right).
$$
When $n=5$ we recover the following data:
\begin{itemize}
\item{the base locus of $\beta^{-1}$, a smooth quintic curve of genus two
$$C \subset \{L_{00}L_{11}-L_{01}L_{10}=0 \} \subset \bP^3;$$}
\item{a divisor on $C$ of degree three, corresponding to 
one of the rulings of the quadric surface.}
\end{itemize}
This is consistent with our previous notation as the space of conics $R_2(X)$ 
admits a natural morphism to $C$ with rational fibers. Indeed, conics in $X$ correspond under $\beta$
to conics $R \subset \bP^3$ incident to $C$ in four points; we take the residual point in the intersection 
of $C\cap \operatorname{span}(R)$. 

\begin{rema}
This yields a geometric explanation for the first assertion of Theorem 28 of \cite{BGW}, in the genus two case.
\end{rema}

We summarize this classical construction:
\begin{cons} (see \cite[Prop.~2.2]{CTSSDI}) \label{cons:line}
Let $X$ be a smooth complete intersection of two quadrics.
Suppose that $F_1(X)$ admits a $k$-rational point $\ell$.
Then $X$ is rational over $k$.
\end{cons}
From this, we obtain
\begin{itemize}
\item{over $k=\bC$, $X$ is rational provided $n\ge 4$;}
\item{over $k=\bF_q$, $X$ is rational provided $n\ge 5$;}
\item{over $k=\bC(B)$, where $B$ is a curve, $X$ is rational
provided $n \ge 6$.}
\end{itemize}
Most of this is contained in \cite[Th.~3.3,3.4]{CTSSDI}, with the
exception of the case of finite fields with $n=5$. 
In this case, $F_1(X)$ is a principal homogeneous space over the
Jacobian of a genus two curve (by Prop.~\ref{prop:DR} or \cite[Th.~4.8]{ReidThesis}) and
thus admits $k$-rational points by Lang's Theorem \cite{LangFF}.

\section{Rationality via points and quadric surface fibrations}
\label{sect:FQS}

\subsection{Double projection}
Let $X\subset \bP^5$ be a smooth complete intersection of two 
quadrics.
\begin{cons} \label{cons:point} \cite[\S 3]{CTSSDI}
For each $x\in X(k)$,
double projection from $x$ induces a quadric surface bundle
$$q:X' \ra \bP^1$$
with six degenerate geometrically integral fibers. 
The relative variety of lines factors
$$F_1(X'/\bP^1) \stackrel{\phi}{\ra} C_x \ra \bP^1$$
where $\phi$ is an \'etale $\bP^1$ bundle, classified by
$\alpha_x \in \Br(C_x)[2]$.
\end{cons}

\begin{prop} \label{prop:diagram}
We have a natural diagram 
$$\xymatrix{
F_1(X'/\bP^1) \ar@{^{(}->}[r] \ar[d]  & F_2(P/\bP^1) \ar[d] \\
C_x 	      \ar@{=}[r] \ar[d]	&       C   \ar[d]  \\
\bP^1         \ar@{=}[r] 	&      \bP^1
}
$$
giving a linear embedding of an \'etale $\bP^1$-bundle into an
\'etale $\bP^3$-bundle.
\end{prop} 
The arrows are induced as follows:
\begin{itemize}
\item{for the bottow row
$$\bP^1 = \bP(T_x\bP^5/T_xX) = \bP(\Gamma(\cI_X(2)) = \bP^1$$
by taking tangent spaces to the quadric hypersurfaces $\{Q_p\}$ in the
pencil;}
\item{for the top row, given $p\in \bP^1$ the threefold
$$\bP^4_p \cap Q_p$$
is a cone over the surface $q^{-1}(p)$ whence
$$F_1(q^{-1}(p)) \hookrightarrow  F_2(\bP^4 \cap Q_p) 
\hookrightarrow F_2(Q_p);$$}
\item{the middle row is induced by the functoriality of 
Stein factorization.}
\end{itemize}
A degree computation shows that the $\bP^1$-fibers are linearly embedded
in the $\bP^3$-fibers.

Proposition~\ref{prop:diagram} yields the compatibility of 
the pairs $(C_x,\alpha_x)$ with the pair $(C,\alpha)$
introduced in Section~\ref{subsect:QP}:
\begin{coro}
The curve $C_x$ and Brauer class $\alpha_x$ are independent of $x$.
The index of $\alpha$ equals two whenever $X$ admits a rational point.
\end{coro}

\begin{cons} 
Retain the notation of Construction~\ref{cons:point}.
Note that $X'$ is rational when $\alpha=0$.
\end{cons}
As a consequence, we obtain that $X$ is rational \cite[Th.~3.4]{CTSSDI}
\begin{enumerate}
\item{over $k$ a $C_i$-field when $n\ge 2^{i+1}+2$;}
\item{over $k$ a $p$-adic field when $n\ge 11$.}
\end{enumerate}

\subsection{Geometric analysis} 
\label{subsect:GA}
We first elaborate the geometric implications of our rationality construction.

Let $X\subset \bP^5$ denote a smooth complete intersection of two quadrics
over $k$ and suppose 
there exists an $x\in X(k)$ such that there exist distinct lines
$$x \in \ell_1,\ell_2,\ell_3,\ell_4 \subset \bar{X}=X_{\bar{k}}$$
with $N_{\ell_i/\bar{X}} \simeq \cO_{\bP^1}^2$.
Double projection from $x$ induces a rational map
$$X \dashrightarrow \bP^1$$
defined away from the lines. We resolve indeterminacy
$$\begin{array}{rcccc}
	&		& \wX &  	 &    \\
        & \swarrow	&     & \searrow &    \\
   X    &               &     &          & X' \\
        &		&     &          & \downarrow {\scriptstyle q} \\
        &               &     &          & \bP^1
\end{array}
$$
where $\wX \ra X$ is obtained by blowing up $x$ and the proper transforms of 
$\ell_1,\ldots,\ell_4$, and $\wX \ra X'$ blows down the exceptional divisors
over the lines along the opposite rulings. The morphism $X'\ra \bP^1$ is a quadric
surface bundle. Let $Y'\subset X'$ denote the
proper transform of the exceptional divisor over $x$, a conic bundle over $\bP^1$
with three degenerate fibers.

Suppose that the class $\alpha \in \Br(C)[2]=0$ so that 
$$\phi: F_1(X'/\bP^1) \ra C$$
admits a section. Then we may express
$$F_1(X'/\bP^1) = \bP(\cE)$$
where $\cE$ is a vector bundle of rank two and odd determinant \cite{DR}. Note that
it follows immediately that $[\Pic^1(C)]=0$ as a principal homogeneous space under 
$A=\Pic^0(C)$.  Thus we may normalize so that $\deg(\cE)=5$. 

The structure of sections of $\phi$ is well-known: For each $\cL \in \Pic^0(C)$, elements of
$$\bP(\Gamma(C,\cE\otimes \cL))$$
yield sections of the projective bundle. Riemann-Roch shows these are parametrized by
a Zariski $\bP^2$-bundle over $\Pic^0(C)$. Their proper transforms in $X$ are twisted
cubic curves containing $x$; these are birational to a Zariski $\bP^2$-bundle over $F_1(X)$
by:
\begin{prop} \label{prop:twistedcubic}
Twisted cubic curves on $X$ are residual to a line
$\ell \subset X$ in the three-plane they span.
Thus the twisted cubics are birational to a Zariski $\Gr(2,4)$-bundle over $F_1(X)$.
Those passing through a fixed point form a Zariski $\bP^2$-bundle. 
\end{prop}
In particular, $X$ admits a twisted cubic over $k$ if and only if it admits a line over $k$.

Lines in $X$ disjoint from the lines $\ell_1,\ldots,\ell_4$ correspond to sections of
$$\Gamma(C,\cE\otimes \cL), \quad \cL \in \Pic^{-1}(C);$$
Riemann-Roch shows there is a unique such section for generic $\cL$.  

We summarize this discussion as follows:
\begin{theo} \label{theo:line}
Let $X\subset \bP^5$ be a smooth complete intersection of two quadrics. 
Suppose that
\begin{itemize}
\item{$X(k)\neq \emptyset$;}
\item{the class $\alpha \in \Br(C)$ vanishes.}
\end{itemize}
Then $X$ is rational and contains a line.
\end{theo}
\begin{proof}
We may assume that $k$ is infinite -- we have already seen that $X$ always admits a line over finite
fields. In this case, Remark~3.28.3 of \cite{CTSSDI} applies to show that $X$ is unirational, and
thus $X(k)$ is Zariski dense. Hence we may find a point $x$ satisfying the genericity assumptions.
This allows us to apply the construction above.
\end{proof}

\begin{rema}
The presence of a line on $X$ allows us to realize $C \subset \bP^1 \times \bP^1$
as a bidegree $(2,3)$ curve (see Section~\ref{sect:ratbylines}).
It is perhaps surprising that the triviality of $\alpha$ forces the existence of odd
degree cycles on $C$.  
\end{rema}

\section{Variety of lines as a cocycle}
\label{sect:cocycle}
Let $C\ra \bP^1$ denote the genus two curve associated with $X$.
Note that $\Pic^0(C)$ and the Albanese of the variety of lines $F_1(X)$
are isomorphic by work of X.~Wang \cite{wang} -- we call this abelian surface $A$.
Moreover, we obtain identifications of principal homogeneous
spaces over $A$:
$$2[F_1(X)]=\Pic^1(C).$$
The class $[F_1(X)]$, as a principal homogeneous space over $A$, has order dividing
four by \cite[Th.~1.1]{wang}.

The associated morphism
$$\kappa:\Sym^2(F_1(X)) \ra \Pic^1(C)$$
has nodal Kummer surfaces as fibers. Over $C\subset \Pic^1(C)$,
we obtain the reducible conics $R_{2,red}(X) \subset R_2(X)$. Thus we have a 
natural factorization
$$
\begin{array}{ccc}
R_{2,red}(X) & \hookrightarrow & F_2(X) \\
                      & \searrow & \downarrow \\
                       &                &    C
\end{array}
$$
realizing the Kummer fibers as $16$-nodal quartic surfaces in the Brauer-Severi
fibration associated with $\alpha$.

We record some basic observations. 
\begin{itemize}
\item If $[F_1(X)]$ is trivial then $X$ admits a line over $k$ and is rational by 
Construction~\ref{cons:line}. 
\item The Brauer class $\alpha \in \Br(C)[2]$ is determined geometrically by the class
$[F_1(X)]$.
\item The class $[F_1(X)]$ has index dividing four when $X(k) \neq \emptyset$; indeed,
the lines incident to a point are defined over a quartic extension.
\end{itemize}
Thus the class $[F_1(X)]$ is the fundamental object for our purposes.

\begin{ques}  \label{ques:CT} \
\begin{enumerate}
\item[i.]{If
$[F_1(X)]$ has order four must $X$ be irrational over $k$?}
\item[ii.]{Are
there examples where $X$ is rational and $[F_1(X)] \neq 0$?}
\end{enumerate}
\end{ques}
The second question grew out of discussions with Colliot-Th\'el\`ene.

\begin{prop}
Consider the following statements:
\begin{enumerate}
\item{$[F_1(X)]$ has order two;}
\item{$[F_1(X)]$ has index two, i.e., $F_1(X)$
 admits a rational point over a quadratic extension of $k$;}
\item{$R_2(X) \neq \emptyset$;}
\item{$C$ admits a rational point.}
\end{enumerate}
Then we have
$$(3) \Rightarrow (4) \Rightarrow (1), (2) \Rightarrow (1).$$
\end{prop}
Below we analyze the geometry of each case, with a view toward understanding
whether these implications are strict.

\begin{rema}
Suppose that $k$ is a $C_1$-field and $C(k)\neq \emptyset$. Then $X$ admits a conic $D\subset X$
defined over $k$. Projecting from $D$ gives a fibration
$$\Bl_D(X) \ra \bP^2$$
in conics, with a quartic plane curve as degeneracy.  

Using conic bundles for rationality constructions
is problematic, as we cannot expect to find sections when the base is a surface.  Such sections would yield extra
divisor classes in the total space, which would have to be singular by the Lefschetz hyperplane theorem. 
(Our $X$ is rational over $\bar{k}$ but not because of the conic fibration.)
\end{rema}

\section{Odd degree curves, secants, and rationality}
\label{sect:odddegree}
\begin{theo} \label{theo:secant}
Let $X \subset \bP^5$ be a smooth complete intersection of
two quadrics. Suppose that $X$ admits a smooth geometrically
connected curve $R \subset X$ of odd degree. Then $X$ contains
a line and is rational.
\end{theo}

\subsection{Background on secants}
Let $R \subset \bP^5$ be a smooth projective geometrically connected curve
of degree $d$ and genus $g$. Let $\Sec(R)$ denote the closure of the secants 
of $R$. We have a diagram
$$
\xymatrix{
B_2 \ar[r]^{\beta} \ar[d] &  \Sec(R)   \ar@{^{(}->}[r] & \bP^5 \\
\Sym^2(R)   & &
}
$$
where $\Sym^2(R)$ is the symmetric square of $R$, $B_2\ra \Sym^2(R)$ is $\bP^1$-bundle
parametrizing lines parametrized by pairs of points on $R$, and $\beta$ is the canonical
morphism.
\begin{prop} \cite[Th.~4.3]{Dale}
Assume that the span of $R$ has dimension at least four.
Then $\beta$ is birational onto its image and
$$\deg(\Sec(R))=\binom{d-1}{2} - g.$$
\end{prop}
The formula has the following interpretation: Secants passing through
a plane $P$ correspond to nodes of the image of the projection
$$\pi_P:R \ra \bP^2;$$
the formula is the difference between the genus of $R$ and the arithmetic genus
of $\pi_P(R)$.

\begin{prop} \cite[1.4, 1.5]{Bertram}
Assume that length four subschemes of $R$ impose independent
conditions on linear forms. Then $\Sec(R)$ is smooth away from $R$
and its projective tangent cone at $r \in R$ is
obtained by the projection of $R$ from the tangent line of $R$
at $r$.
\end{prop}
Under our assumptions projection of $R$ from the tangent line
at $r$ maps it isomorphically onto a degree $d-2$ curve in 
$\bP^3$. Thus $\Sec(R)$ has multiplicity $d-2$ along $R$.

\subsection{Analysis in complete intersections of quadrics}
Let $X \subset \bP^5$ denote a smooth complete intersection of quadrics.
We denote its hyperplane class by $h$.
Consider 
$$X'=\{(x,\ell): x\subset \ell  \} \subset X \times F_1(X)$$
which has the following properties:
\begin{itemize}
\item{$\pi_1:X' \ra X$ is generically finite of degree four;}
\item{$\pi_1$ is branched over a hypersurface in $X$ of degree eight.}
\end{itemize}
It is clear that four lines pass through a generic point of $X$.
For the second assertion, fix a generic line $\ell \in F_1(X)$ so
the resulting
$$\ell':=X'\times_X \ell \ra \ell$$
has two disjoint components $\ell$ and $C$, where $C \ra \ell$
is a degree-three cover from a curve of genus two branched over eight
points.  

\begin{prop} \label{prop:gensecant}
Let $R\subset X$ be smooth and geometrically connected
of degree $d\ge 3$ and genus $g$. Assume that
\begin{itemize}
\item{$\Sec(R)$ has isolated
singularities of multiplicity $d-2$ at the generic point $r$
of $R$, and }
\item{finitely many secants to $R$ are contained in $X$.}
\end{itemize}
Let $\Sigma_R$ denote the sum of the secants to $R$ in $X$
counted with multiplicity. Then we have
$$(d-2)R + \Sigma_R \equiv  \left(\binom{d-1}{2} - g\right) h^2$$
in the Chow group of $1$-cycles of $X$.
\end{prop}
This is similar to results of M.~Shen \cite{Shen} for cubic threefolds.
\begin{proof}
We assume for the moment that $R$ spans a subspace of dimension at least four.
Since $[\Sec(R)]\equiv (\binom{d-1}{2}-g)h^2$ and $\Sec(R)$ intersects
$X$ in pure dimension one, we just need to interpret the terms of the 
intersection. The computation of tangent cones above means that
$\Sec(R) \cap X$ has multiplicity  $m \ge (d-2)$ along $R$.

Consider the induced morphism
$$R':=X'\times_X R \ra F_1(X);$$
its arithmetic genus $p_a(R')$ satisfies
$$2p_a(R') -2 = 4(2g-2)+8d$$
whence $p_a(R')= 4g + 4d - 3$.
We claim $R'$ maps birationally onto its image $D$. Indeed, suppose
that $[\ell]\in D$ has two points $r'_1,r'_2 \in R'$ lying over it.
Let $r_1,r_2 \in R$ be their images in $R$; the line $\ell$
contains these points and is therefore a secant to $R$. 

An intersection computation shows that the class $[D] \in N_1(F_1(X))$
(one-cycles up to numerical equivalence) is equal to $d\Theta$. 
Curves in this class have arithmetic genus $d^2+1$ and secants
to $R$ correspond to double points of the induced map
$R' \ra D.$ There are
$$p_a(d\Theta) - p_a(R') = d^2 - 4d -4g + 4$$
such points, whence
$$\deg(\Sigma_R)= d^2 -4d -4g+4.$$
Comparing degrees
$$m d + d^2 -4d-4g+4 = 2(d-1)(d-2) - 4g$$
we conclude that $m=d-2$ and the formula follows.

Now assume that $R$ spans a subspace of dimension three. 
If $R$ is a twisted cubic then it is residual to a line $\ell$ in a complete
intersection of linear linear forms; $\ell$ is the unique secant
to $R$ in $X$ so that $\Sigma_R=\ell$. Then we obtain
$$ R + \ell \equiv h^2$$
which is consistent with our general formula. If $R$ is an elliptic
quartic curve then 
$$R \equiv h^2$$
which is also consistent with the formula.
\end{proof}

\begin{prop} \label{prop:reducetoinfinite}
Let $R\subset X$ and assume that $\Sec(R)$ does not have
isolated singularities of multiplicity $d-2$ at the generic point $r$
of $R$. Then there are infinitely many secants to $R$ contained in $X$.
\end{prop}
\begin{proof}
Under our assumptions, there exist points $p,q \in R \setminus \{r\}$ such that
the secant $\ell(p,q)$ contains $r$. It follows that the line
$\ell(p,q)$ meets $X$ with multiplicity at least three and thus
is contained in $X$. Varying over $r \in R$, we obtain a curve
$$B \ra F_1(X)$$
such that $R$ is contained in the ruled surface ${\bP}(S^*|B) \subset X$
associated with $B$. Here $S$ is the universal line sub-bundle on
the Grassmannian $\Gr(2,6)$.
\end{proof}

\begin{prop} \label{prop:infsec}
Suppose that $R \subset X$ admits infinitely many secants on $X$.
If $R$ has odd degree then $X$ admits a cycle of lines of odd degree.
\end{prop}
\begin{proof}
Let $B^{\nu} \ra B \subset F_1(X)$ denote the normalization
such that $R$ factors through
a ruled surface $\bF:=\bP(S^*|B^{\nu}) \ra X$. Writing 
$$b=(\Theta\cdot B)_{F_1(X)},$$
we see that $[\bF]=bh$ and, cutting by hyperplanes,
we find that $\bF\ra B^{\nu}$ admits sections
of class $bh^2$. These both have degree $4b$. If $\bF \ra B^{\nu}$ admits
a multisection $R$ then $B^{\nu}$ admits a zero-cycle
of odd degree as well. Indeed, take the difference of $R$ and a suitable
multiple of the sections mentioned above. Thus $F_1(X)$ admits
such a cycle as well.
\end{proof}

\subsection{Proof of Theorem~\ref{theo:secant}}
If $X$ contains a line then there is nothing to prove. Otherwise,
Proposition~\ref{prop:gensecant} covers the `generic' case. The
number of secants with multiplicities is
$$\deg(\Sigma_R)= d^2 -4d -4g+4,$$
which is odd whenever $d$ is odd. However, $F_1(X)$ has a point
whenever it admits a cycle of odd degree. When the 
genericity hypotheses fail to be satisfied,
Proposition~\ref{prop:reducetoinfinite} puts us in the case where
there are infinitely many secants to $R$ on $X$. Here we apply 
Proposition~\ref{prop:infsec} to conclude the result.

\section{Two skew lines}
\label{sect:skew}
The motivating question of this section is: Let $X \subset \bP^5$ denote a smooth
complete intersection of two quadrics over a field $k$. Suppose that 
$X$ admits a line $\ell \subset X$ defined over a quadratic extension $L/k$,
not containing a $k$-rational point of $X$. In other words, $\ell$ and its 
Galois conjugate $\ell'$ are disjoint. When does it follow that $X$ is rational? 
\begin{exam}
Over $k=\bR$ we are guaranteed lines over quadratic extensions.
If $X(\bR)= \emptyset$ then pairs of conjugate lines are
necessarily skew but $X$ is not rational.
So the existence of two conjugate lines {\em per se} does
not imply rationality.
\end{exam}

\begin{cons} \label{cons:dP6} (cf. \cite[\S 4,6]{CTSSDI})
Suppose that $X$ admits a pair of skew lines $\ell$ and $\ell'$,
Galois conjugate over
$k$. Express
$$\operatorname{span}(\ell,\ell')\cap X$$
as a quadrilateral $\{\ell,m_1,\ell',m_2\}$. Then
projecting from $\operatorname{span}(\ell,\ell')$
gives a sextic del Pezzo fibration
$$\varpi: Y \ra \bP^1,$$
obtained by blowing down skew lines.
\end{cons}
Our analysis will describe the degeneracy locus and associated
Galois-theoretic data of $\varpi$.

\subsection{Details of the construction}

Let $m_1,m_2 \subset X$ be residual to $\ell$ and $\ell'$ in a codimension-two
linear section of $X$
$$\operatorname{span}(\ell,\ell') \cap X = \ell \cup m_1 \cup \ell' \cup m_2.$$
Consider the hyperplane sections of $X$ containing $\ell,\ell',m_1$, and $m_2$,
which are quartic del Pezzo surfaces with a fixed anti-canonical cycle of lines.
The fiber-wise linear series consists of quadrics vanishing 
along $\ell$ and $\ell'$, which collapse $m_1$ and $m_2$.

\begin{rema}
In this situation, we have a natural rational map
$$\psi: X \dashrightarrow \mathbf{R}_{L/k} \bP^3$$
induced by projection from $\ell$ and $\ell'$. This factors as
\begin{itemize}
\item{blowing up $\ell$ and $\ell'$;}
\item{blowing down the lines $m_1$ and $m_2$ residual to $\ell$ and $\ell'$.}
\end{itemize}
Let $Y$ denote the closed image of $X$. Geometrically, it is a complete intersection of 
three bidegree $(1,1)$ forms in $\bP^3 \times \bP^3$ with two ordinary threefold double
points $y_1,y_2 \in Y$, the images of $m_1$ and $m_2$.
Note that 
\begin{equation} \label{eqn:one}
\deg(Y) = \binom{6}{3} = 20, Y \subset \bP^{12}.
\end{equation}
\end{rema}

The pencil of hyperplane sections
of $X$ containing the cycle of rational curves
$$\ell \cup m_1 \cup \ell' \cup m_2 $$
gives a rational map
$$\varphi: X \dashrightarrow \bP^1$$
that may be factored a follows.

\

\paragraph{\em Step 1} Let $X_1 \ra X$ denote the blow up of
$\ell$ and $\ell'$ with exceptional divisors $\hat{E}$ and $\hat{E}'$,
each isomorphism to $\bP^1 \times \bP^1$. 
Note that 
$$(\hat{E})^3=(\hat{E}')^3=0.$$
The proper transforms $\hat{m}_1$ and $\hat{m}_2$ have 
normal bundles $\cO_{\bP^1}(-1) \oplus \cO_{\bP^1}(-1)$. 
\

\paragraph{\em Step 2} Let $X_2 \ra X_1$ denote the blow up of 
$\hat{m}_1$ and $\hat{m}_2$ with exceptional divisors 
$F_1$ and $F_2$. We have 
$$(F_1)^3 = (F_2)^3 = 2$$
from the normal bundle computation.
Let $\tilde{E}$ and $\tilde{E}'$ denote the 
proper transforms of $\hat{E}$ and $\hat{E}'$. 

\

\paragraph{\em Step 3} Let $g$ denote the pull-back of the hyperplane
class of $X$ to $X_2$. We claim that
$$M = g - \tilde{E} - \tilde{E}' - F_1 - F_2$$
is basepoint free. We analyze it on each exceptional divisor.

We can write
$$\Pic(\tilde{E}) = \bZ g + \bZ e + \bZ f_1 + \bZ f_2,$$
where $e=-\tilde{E}|\tilde{E}$ and $f_i=F_i|\tilde{E}$,
with nonzero intersections
$$g\cdot e = 1, f_1^2=-1, f_2^2=-1.$$
We have $\tilde{E}$ is a sextic del Pezzo surface and $M|\tilde{E}$
induces a conic bundle, which is basepoint free.
We can write 
$$\Pic(F_1) = \bZ g + \bZ \eta$$
with nonzero intersections $g\eta = 1$. 
We have
$$M|F_1 = g - 2g + (g+\eta) = \eta.$$
Thus $F_1$ and $F_2$ are collapsed to $\bP^1$'s in the `opposite' direction.

Note the further nonvanishing intersection numbers
$$g(\tilde{E})^2= g(\tilde{E}')^2 = g(F_1)^2 = g(F_2)^2 = -1 $$
and 
$$\tilde{E}F_1^2 = \tilde{E}F_2^2 = \tilde{E}' F_1^2 = \tilde{E}' F_2^2=-1.$$
Thus we have
$$(g-\tilde{E} - \tilde{E}' - F_1 - F_2)^3 = 4 -3 -3 -3 -3 + 3 +3 + 3 + 3 -2 -2 =0.
$$

\

\paragraph{{\em Step 4}} We interpret this map as flopping $\tilde{m}_1$ and
$\tilde{m_2}$, to get
$$\varphi': X'_1 \ra \bP^1.$$
The corresponding lines $m_1$ and $m_2$ in our pencil of
hyperplane sections are contracted, yielding fibers
that are sextic del Pezzo surfaces.

\

\paragraph{{\em Step 5}} Since 
$$\chi(X'_1)=\chi(X_1)=4-4+2+2=4,$$
assuming $\varphi'$ is degenerate along
fibers with one ordinary double point,
the degeneracy locus $\Delta  \subset \bP^1$
satisfies
$$ 4 = (2-|\Delta|)6 + 5|\Delta| \text{ whence } |\Delta|=8.$$
Our fibration thus has eight singular fibers.

 \
 
\paragraph{{\em Geometric description}}
The Galois representation associated with $\varphi'$ 
$$\rho: \operatorname{Gal}(k(\bP^1)) \ra \fS_2 \times \fS_3$$
may be interpreted as follows. The distinguished quadratic
extension associated to the blow-up realizations over $\bP^2$
is induced by $L/k$, the field of definition of $\ell$ and $\ell'$.

We analyze conics in the fibers of $\varphi'$.
Their proper transforms $Z\subset X$ are incident
to $\ell$ and $\ell'$ but disjoint
from $m_1$ and $m_2$. They are parametrized by a surface $T$ fibered in
conics over a genus-two curve $C'$, which admits a morphism 
$$C' \ra C \hookrightarrow  \Pic^1(C)$$
and a triple cover $C' \ra \bP^1$ with eight degenerate fibers.
It follows that $C'\simeq C$ and the latter curve admits a degree-three
morphism to $\bP^1$.

\subsection{Application to rationality}
\label{subsect:ATR}
Retain the set-up of Construction~\ref{cons:dP6}, which yields a
fibration $\varpi:Y \ra \bP^1$ in sextic del Pezzo surfaces
where $Y$ is birational to $X$ over $k$.
The generic fiber of $\varpi$ is rational over $k(\bP^1)$ if and only if
$\varpi$ admits a section.
Indeed, a sextic del Pezzo surface over a field is rational whenever
it admits a point \cite[Cor.~1 to Thm.~3.10]{ManinIHES}.
It follows that $X$ is also rational over $k$.

However, such a section is the proper transform of a curve $R\subset X$
of odd degree -- the hyperplanes in the pencil meet the
section in one point outside $\{\ell,m_1,\ell',m_2\}$ and $R$ meets 
these lines in pairs of conjugate points.
Assuming that $R$ is smooth, Theorem~\ref{theo:secant} implies that $X$
actually contains a line (cf. Question~\ref{ques:CT}).

\section{Rational curves in higher degrees}
\label{sect:higher}
We have already given the structure of lines, conics, and twisted cubics on a smooth complete
intersection $X\subset \bP^5$ of two quadrics. 
We describe the rational normal quartic curves 
$$R\subset X_{\bar{k}}$$
by making reference to a rational parametrization
$$\rho: \bP^3 \dashrightarrow  X_{\bar{p}},$$
blowing up $C\subset \bP^3$ realized as a $(2,3)$ divisor
in a quadric surface $Q=\{F=0\}\subset \bP^3$ (see Section~\ref{sect:ratbylines}).
The preimage $\rho^{-1}(R)$ is also a rational quartic curve $R'$
with
$$R'\cap C = \{c_1,\ldots,c_8\}.$$
Thus we obtain a rational map
$$M_0(X,4) \dashrightarrow \operatorname{Sym}^8(C).$$

We claim this is generically finite of degree four.
Indeed, we write
$$I_{c_1,\ldots,c_8}(2)=\left< F, G\right>$$
and we have an elliptic quartic curve
$$E=\{F=G=0\}$$
and computing intersections in $Q$ yields
$$(E\cap C)_Q = \{c_1,\ldots,c_8,c_9,c_{10}\}.$$
Each $R'$ arising as above sits in a member of the pencil
$$Q' \in \{ sF + tG =0 \}$$
as a divisor of bidegree $(1,3)$. Rulings of quadrics in this pencil are 
parametrized by $E$, in which we have the equation
$$f_1+3f_2|E = \cO_E(c_1+\cdots+c_8)$$
or equivalently
$$2f_2 = \cO_E(c_1+\cdots+c_8)(-1).$$
This has four solutions with the structure of a principal
homogeneous space over $E[2]$.

The image of the proper transform of $Q'$ is obtained in two
steps
\begin{itemize}
\item{blow up $c_1,\ldots,c_{10}$;}
\item{pinch along the $E$ using the degree two morphism
$$E\ra \bP^1$$
induced by the ruling of $Q$ contracted by $\rho$.}
\end{itemize}
This is obtained by taking a quadric hypersurface section
$$\Sigma \subset X$$ 
containing $R$ and double along $\ell$ -- these form
a linear system of dimension
$$21-2-9-3-6=1.$$
We are using the fact that $N_{\ell/X}=\cO_{\bP^1}^2$ or 
$\cO_{\bP^1}(1)\oplus\cO_{\bP^1}(-1)$
to compute the conditions imposed by insisting that the quadric
is double along $\ell$. Thus $\Sigma$ is uniquely determined by $\ell$
and $R$.

\section{Irrationality via specialization of birational types}
\label{sect:toric}

\subsection{The toric example}
First, we work over an algebraically closed field.

Let $X$ be toric, isomorphic to
$$x_0x_1-x_2x_3=x_2x_3-x_4x_5=0,$$
with ordinary singularities
at the coordinate axes.
We may realize $X$ as the image of the rational map
\begin{align*}
    \bP^3 &\stackrel{j}{\dashrightarrow} X \subset \bP^5 \\
[y_0,y_1,y_2,y_3] &\mapsto [y_0y_1,y_2y_3,y_0y_2,y_1y_3,y_0y_3,y_1y_2]
\end{align*}
that blows up the coordinate points $p_1,p_2,p_3,p_4$ and 
collapses the lines $\ell_{12},\ldots,\ell_{34}$
joining them to singularities.
Let $\beta: \wX \ra X$ denote the resolution obtained by blowing up 
the coordinate axes, or equivalently, blowing up 
$p_1,\ldots,p_4$ and then the proper transforms of $\ell_{12},\ldots,\ell_{34}$.

Consider the the graph $X'$ of the standard Cremona transformation
\begin{align*}
  \iota:\bP^3 & \dashrightarrow  \bP^3 \\
  [y_0,y_1,y_2,y_3] & \mapsto  [1/y_0,1/y_1,1/y_2,1/y_3],
\end{align*}
whose indeterminacy is given by the scheme
$$\{y_1y_2y_3=y_0y_2y_3=y_0y_1y_3=y_0y_1y_2=0\}.$$
Letting $z_0,z_1,z_2,z_3$ be projective coordinates on the target projective space,
we have
$$X' = \{y_0z_0=y_1z_1=y_2z_2=y_3z_3 \} \subset \bP^3 \times \bP^3.$$
It has ordinary singularities at 12 points 
$$([1,0,0,0],[0,1,0,0]),([0,1,0,0],[1,0,0,0]),\ldots, $$
with three over each of the coordinate axes in the $\bP^3$ factors.

We have a diagram
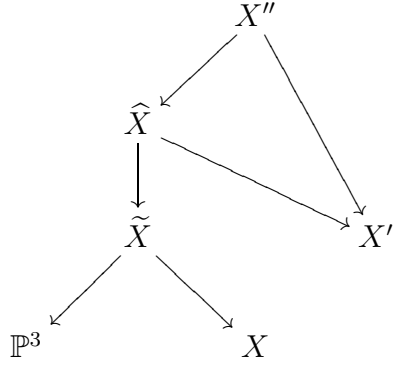
\[
\xymatrix{
& & X'' \ar[ld]  \ar[rdd]  &  \\
& \hX \ar[d] \ar[rrd] &    &   \\
& \wX \ar[rd] \ar[ld]   &     & X' \\
\bP^3 & & X                 &   
}
\]
with arrows as follows:
\begin{itemize}
\item{$\hX \ra \wX$ blows up the proper transforms of the lines 
$\ell_{12},\ldots,\ell_{34}$;}
\item{$X'' \ra X'$ blows up the 12 ordinary double points;}
\item{$X''\ra \hX$ blows up the proper transforms of the 12 lines 
$L_{i;jk}$
in the exceptional divisors of $\wX \ra \bP^3$ connecting the 
proper transforms of $\ell_{ij}$ and $\ell_{ik}$;}
\item{$\hX \ra X'$ contracts the proper transform of the $L_{i;jk}$.}
\end{itemize}

In addition to the toric action,
there is an action of a group 
$$H \simeq \fS_4 \times \fS_2 \simeq (\bZ/2\bZ)^3 \rtimes \fS_3,$$
where the semidirect product is by permutation of the factors.
The $\fS_4$ symmetry is the permutation of the coordinates on $\bP^3$;
the $\fS_2$ symmetry is induced by the
standard Cremona transformation on $\bP^3$. Thus the action of $H$ on
$X'$ (and thus its resolution $X''$) is clear from the notation.
The $\fS_4$ action is clearly regular on $X$ as well; the regularity
of the Cremona involution on $X$ (and thus $\wX$)
reflects the commutativity of the diagram
\[
\xymatrix{
\bP^3 \ar@{-->}[r]^j \ar@{-->}[d]_{\iota} & X \ar[d]^{\phi} \\
\bP^3 \ar@{-->}[r]^j & X
}
\]
Here $\iota$ is the Cremona involution, $j$ the birational map between 
$\bP^3$ and $X$, and $\phi$ is the linear transformation
$$\phi(x_0,x_1,x_2,x_3,x_4,x_5)=(x_1,x_0,x_3,x_2,x_5,x_4)$$
that preserves $X$.

\subsection{The variety of lines}
\begin{prop}
\label{prop:ll}
Consider the variety of lines $F_1(X)$.
\begin{itemize}
\item{$\deg(F_1(X))=32$ ;}
\item{it has eight components isomorphic to $\bP^2$: 
$P_1,P_2,P_3,P_4$ and $P_{123},P_{124},P_{134},P_{234}$ ;}
\item{it has four components isomorphic to the degree six del Pezzo
surface: 
$S_1,\ldots,S_4$ ;}
\item{$S_1$ meets $P_2,P_{123},P_3,P_{134},P_4,P_{124}$
in a hexagonal anticanonical cycle;}
\item{$P_1$ meets $S_2,S_3,S_4$ in a
triangular anticanonical cycle;}
\item{$P_{123}$ meets $S_1,S_2,S_3$ in a  
triangular anticanonical cycle.}
\end{itemize}
\end{prop}
\begin{proof}
First, note that $X$ contains eight planes
$$\{x_0=x_2=x_4=0\}, \{x_0=x_2=x_5=0\} ,\ldots ,\{x_1=x_3=x_5=0\},
$$
each containing three of the ordinary double points.
The lines in these planes are the eight $\bP^2$'s in $F_1(X)$. These may be interpreted
via $j$ as:
\begin{itemize}
\item{$P_i$ -- corresponding to lines in the exceptional divisors over the
$p_i$;}
\item{$P_{ijk}$ -- corresponding to conics in $\bP^3$ containing $p_i,p_j,p_k$.}
\end{itemize}
A direct computation shows that $F_1$ is singular
along those lines meeting the six singularities. These form $24$
lines in $F_1(X)$, four associated with each singularity and three lying
on each of the eight planes.
We write
\begin{itemize}
\item{$S_i$ -- corresponding to the lines in $\bP^3$ through $p_i$.}
\end{itemize}
These are parametrized by the blow-up of $\bP^2$ at through noncollinear points, 
a sextic del Pezzo surface, cf.~\cite[Th.~1.10]{AK}. Note that the distinguished hexagon of lines in $S_1$ 
coincides with its intersections with the six planes indicated.  

On the other hand, a computation shows that $F_1(X)$ is smooth of dimension two
at all lines {\em not} meeting to the singularities. In particular, $F_1(X)$ has pure dimension two.   

The statement on the degree of $F_1(X)$ can be obtained via Schubert calculus on the Grassmannian
$\Gr(2,6)$.  It reflects the fact that $\cO_{\Gr(2,6)}(1)|F_1(X)$ is four times the
principal polarization. 

It remains to show that the enumerated lines cover all the lines on $X$. The sum of the degrees of the 
$S_i$ equals $24$; the sum of the degrees of the $\bP^2$ components is $8$. Thus we conclude $F_1(X)$
is the union of these $12$ surfaces.   
\end{proof}

Our notation is chosen compatibly with the action of $\fS_4\times \fS_2$. The first factor permutes the indices. The second 
interchanges $P_i$ and $P_{jkl}$ where $\{i,j,k,l\}=\{1,2,3,4\}$.

\subsection{Galois cohomology}
Rationality of a 3-dimensional torus $T$ is governed by the Galois action on its lattice of characters 
$\mathfrak X^*(T)$; the action factors through a subgroup $G\subset H$, where 
$$
H\simeq \mathfrak S_4\times \mathfrak S_2 \subset \mathrm{GL}_3(\bZ) = \Aut(\mathfrak X^*(T)). 
$$
The main result of \cite{kun} is:
\begin{itemize}
\item A  3-dimensional torus is nonrational if and only if $G$ contains a subgroup 
$U_1\simeq \mathfrak  S_2\times \mathfrak S_2$, see \cite[Section 3]{kun}. 
\end{itemize}
In particular, all  $T$ with $G\subset \mathfrak S_4$ are rational \cite[Lemma 1]{kun}. 
Explicitly, the generators of $H$ are given as matrices 
$$
a:=\left( \begin{array}{rrr} 0 & 1 & 0 \\ 0 & 0 & 1\\ -1 & -1 & -1\end{array}   \right), \,
b:=\left( \begin{array}{rrr} 0 & 1 & 0 \\ 1 & 0 & 0\\ 0 & 0 & 1\end{array}   \right), \,
c:= \left( \begin{array}{rrr} -1 & 0 & 0 \\ 0 & -1 & 0\\ 0 & 0 & -1\end{array}   \right),
$$
These satisfy the relations
$$
a^4 =b^2=(ab)^3=1, c^2=1,
$$ 
where $c$ is central.  The generators of $U_1$ are given in \cite[Thm. 1]{kun}, they are (modulo a change of basis)
$$
a^2 = \left( \begin{array}{rrr} 0 & 0 & 1 \\ -1 & -1 & -1\\ 1 & 0 & 0\end{array}   \right), \quad
bc= \left( \begin{array}{rrr} 0 & -1 & 0 \\ -1 & 0 & 0\\ 0 & 0 & -1\end{array}   \right),
$$
A full list of nonrational three-dimensional tori is in \cite[Thm. 1]{kun}. 

\begin{rema}
\label{rema:lines}
Assume that the Galois group acts transitively on the set of 4 coordinate points $p_1,\ldots, p_4\in \bP^3$, giving rise to $X$ as above.
Then $X$ is rational but does not contain lines defined over $k$, by the description in Proposition~\ref{prop:ll}.  
\end{rema}

\begin{rema}
The group $U_1$ may be realized as a subgroup of $H$, e.g., by taking $c$ as the generator of the 
second factor, $a=(1 2 3 4)$ and $b=(1 2)$. This group does not fix any of the irreducible components of $F_1(X)$.
On the other hand, the Cremona involution acts on the $S_i$ by taking inverses and thus stabilizes points not on hexagon, i.e., lines 
on $X$ not meeting singularities.  
\end{rema}

\subsection{Construction of nonrational examples} \label{subsect:nonrat}
Since the action of $H$ is regular on the toric variety $X$, we can use it to obtain a twisted model of ${}^{\rho}\!X$ and the corresponding torus over a nonclosed field, provided there is a representation $\rho:\Gal(k) \ra H$.
For example, let $K$ be a cubic extension over $k$ and $L$ and quadratic extension over $K$; assume that the Galois closures
of $K$ and $L$ over $k$ have Galois groups $\fS_3$ and $H$ respectively. Let $z$ be an indeterminate in $K$ and 
$\theta\in K$ a primitive element; the equations
$$\Tr_{K/k}(z)=\Tr_{K/k}(\theta z) =0$$
are independent. Let $x$ be an indeterminate in $L$ satisfying $\Nm_{L/K}(x)=z$. The equations
$$\Tr_{K/k}( \Nm_{L/K}(x))=\Tr_{K/k}(\theta \Nm_{L/K}(x))=0$$
are homogeneous and define a locus on $\bP (\mathbf{R}_{L/k}\bA^1) \simeq \bP^5$ geometrically isomorphic to
$X$. 

The resulting variety is one of the minimal nonrational toric threefolds considered in \cite[Ex.~6.2.1]{vosk}.

\subsection{Specialization of birational types}
We will use a special case of \cite[Th.~16]{KT}:
\begin{theo} \label{theo:KT}
Let $k$ be a field, $K=k((\tau))$, $B=\operatorname{Spec}(k[[\tau]])$, and 
$\cX\ra B$ a flat projective morphism from a scheme smooth over $k$ with closed fiber
$\cX_0$. We assume that
\begin{itemize}
\item the generic fiber $\cX_{\tau}$ is smooth and rational (resp.~stably rational) over $K$;
\item $\cX_0$ is geometrically irreducible and reduced;
\item $\cX_0$ is singular along a subscheme $Y$ that is smooth over $k$;
\item the blowup $\beta: \Bl_Y(\cX_0) \ra \cX_0$ resolves the singularities of $\cX_0$;
\item the exceptional divisor $D$ of $\beta$ is smooth over $Y$ and rational over $k(Y)$.
\end{itemize}
Then $\cX_0$ is rational (resp.~stably rational) over $k$.
\end{theo}
Our assumptions mean that the proper transform of $\cX_0$ meets the exceptional divisor of $\Bl_Y(\cX)$
transversally.
\begin{proof}
It suffices to verify that the pair $(\cX,\cX_0)$ has $B$-rational singularities. Our argument is an extension
of \cite[Ex.~13]{KT}. 

Let $\cX'=\Bl_Y(\cX)$ with exceptional divisor $E\simeq \bP(N_{Y/\cX})$; $\cX'_0$ is a normal crossings
divisor and write $D$ for the intersection of $E$ with the proper transform of $\cX'_0$, i.e., the exceptional
locus of $\beta$. 
Using \cite[\S 4]{KT}, we write
\begin{align}
\partial_{\cX'_0}(\cX') = &[E \ra \cX'_0] +[\Bl_Y(\cX_0)\ra \cX'_0] - [D \times \bA^1\ra \cX'_0] \\
= &[\Bl_Y(\cX_0)\ra \cX'_0],
\end{align}
where the cancellation comes from definition of the Burnside group and 
the fact that $E$ and $D$ are both rational over $k(Y)$. We conclude that
$$\partial_{\cX_0}(\cX) = [\cX_0^{\operatorname{smooth}} \hookrightarrow \cX_0],$$
the desired condition on the singularities.
\end{proof}

We apply this to the examples of Section~\ref{subsect:nonrat}:
\begin{theo}
Let $k$ be a function field of a complex curve, $\cX_0$ a nonrational toric complete intersection 
of two quadrics over $k$, and $\cX \ra B$ a deformation of $\cX_0$ with smooth total space. Then 
the generic fiber of $\cX_{\tau}$ is not (stably) rational over $k((\tau))$. 
\end{theo}
Note that we can easily choose the deformation so that $\cX\ra B$ admits sections, e.g., by choosing the
deformed quadratic equations to vanish at a smooth point of $\cX_0$. This construction yields examples over function
fields of complex surfaces that are irrational yet admit rational points. 
\begin{coro} \label{coro:main2}
There exist examples of smooth complete intersections of two quadrics in $\bP^5$ over $\bC(t,\tau)$ that
are not rational (or stably rational) but admit rational points.
\end{coro} 
\begin{proof}
Here $\cX_0$ is the twisted form of our toric complete intersection of two quadrics. We take $Y$ to be its six ordinary double points. A deformation that is versal for these singularities has smooth total space. Thus
$D\ra Y$ is a smooth quadric surface fibration over a complex curve. The Tsen-Lang theorem shows that any such fibration is rational, so the hypotheses of Theorem~\ref{theo:KT} are satisfied.   
\end{proof}

\begin{rema}
Toric degenerations of Fano threefolds have attracted attention in connection 
with mirror symmetry and the theory of Landau-Ginzburg models.
We expect that the technique presented here will permit the construction 
of numerous nonrational but geometrically rational smooth Fano threefolds.
\end{rema}

\section{Irrationality via decomposition of the diagonal}
\label{sect:irr}

Here we establish the following result, which answers a question
of Colliot-Th\'el\`ene from 2005:

\begin{theo} 
\label{theo:main}
There exist smooth complete intersections of two quadrics $X\subset \bP^5$
over the field $k=\bC(t)$ that fail to be stably rational over $k$.
\end{theo}

The remainder of this section is devoted to the proof. The idea is to view the quadric surface bundle over $\bP^1$, over $k=\mathbb C(t)$,
as a quadric surface bundle over $\bP^1\times \bP^1$, over $\bC$. If $X$ were rational over $k$ then the associated fourfold would be rational over $\bC$.  
The quadric surface bundle degenerates along a curve of bidegree $(6,6)$. To this we will apply the specialization method of 
\cite{voisin} or \cite{ct-pirutka}, as in \cite{HPT16}. This entails
two steps:
\begin{itemize}
\item exhibit a specialization with nontrivial unramified Brauer group and 
\item show that the singularities of the specialization are mild.
\end{itemize}
Throughout, we make reference to the geometric analysis
in Section~\ref{subsect:GA}.

\subsection{The case of function fields}
We now assume that $k=\bC(t)$.
We have $X(k)\neq \emptyset$ by the Tsen-Lang Theorem; indeed, rational points
are Zariski dense as $X$ is rationally connected over the function field of a curve. (See also the unirationality results in \cite[Prop.~2.3]{CTSSDI}.)

We write down representative models for projective equations of models over $\bP^1$.
The complete intersection of two quadrics sits
$$\cX \subset \bP(\cO_{\bP^1} \oplus \cO_{\bP^1}(-1)^{\oplus 3}
 \oplus {\cO_{\bP^1}(-2)}^{\oplus 2}),
 $$
where the first summand arises from the section associated with
$x\in X(k)$ and the first four summands arise from the tangent space $T_xX$.
Fix $\{x_0\}$, $\{x_1,x_2,x_3\}$, and $\{x_4,x_5\}$
to be weighted variables corresponding to the summands.
We write equations
$  F = G =0 $
where
$$ F = 2x_0x_4 + c_{11}x_1^2 + \cdots + 2l_{14}x_1x_4+\cdots + 
q_{44}x_4^2 + 2q_{45}x_4x_5+q_{55}x_5^2 $$
and
$$ G = 2x_0x_5 + d_{11}x_1^2 +\cdots + 2m_{14}x_1x_4+\cdots + 
r_{44}x_4^2 + 2r_{45}x_4x_5+r_{55}x_5^2 $$
with the $c_{ij}$ and $d_{ij}$ of degree $d$, the $l_{ij}$
and $m_{ij}$ of degree $d+1$ and the $q_{ij}$ and $r_{ij}$
of degree $d+2$.
The projection from $T_xX$ is given by $[x_4,x_5]$; write
$x_5=tx_4$ which we'll take as the second grading.
The elimination is obtained by taking $tF-G$, substituting
$x_5=tx_4$, and then re-homogenizing the $t$ variable. Thus we
obtain a quadric surface bundle
$$\cX' \subset  \bP(\cO^3_{\bP^1\times \bP^1} \oplus \cO_{\bP^1 \times \bP^1}(-1,-1))
 \ra \bP^1 \times \bP^1$$
with equation
$$L_{11}x_1^2 + 2L_{12}x_1x_2+\cdots + 2Q_{14}x_1x_4+\cdots+
C_{44}x_4^2=0$$
associated with the symmetric matrix
$$A = \left( \begin{matrix}
   L_{11} & L_{12} & L_{13} & Q_{14} \\
   L_{12} & L_{22} & L_{23} & Q_{24} \\
   L_{13} & L_{23} & L_{33} & Q_{34} \\
   Q_{14} & Q_{24} & Q_{34} & C_{44} \end{matrix}
\right)$$
where the $L_{ij}$ are bidegree $(d,1)$, the $Q_{ij}$ bidegree $(d+1,2)$
and $C_{44}$ bidegree $(d+2,3)$. The degeneracy locus $D\subset \bP^1 \times \bP^1$
has bidegree $(4d+2,6)$.

Conversely, given a symmetric matrix $A$ of forms with the prescribed bidegrees,
we may reverse the process to recover $\cX$ from $\cX'$.
The construction depends on 
$$6\cdot 2(d+1) + 3\cdot 3(d+2) + 4(d+3) - (9+1+3\cdot 4) - 6 = 25d + 14$$
parameters. However, a generic form of bidegree $(4d+2,6)$ depends on
$$7(4d+3) - (1+6) = 28d + 14$$
parameters.  

The model of $C$
$$\cC \ra \bP^1 \times \bP^1$$
is a double cover branched over a degree-six multisection. The
unramified element $\alpha \in \Br(\cC)[2]$ governs the rationality of
the generic fiber.

If $d=0$ then $\cC \ra \bP^1 \times \bP^1$ is a double cover branched
over a curve of bidegree $(2,6)$ which has trivial Brauer group; here
the fibration $q$ must have a section. The case $d=1$ yields a
bidegree-$(6,6)$ degeneracy locus -- we focus on this.
However, note that the determinantal condition gives a codimension {\bf three}
locus in the parameter space of $(6,6)$ forms. 

\subsection{Application to Theorem~\ref{theo:main}}

Specialize to
$$ 2E \cup F_1 \cup F'_1 \cup F_2 \cup F'_2$$
where $E$ is bidegree $(2,2)$, $F_1$ and $F'_1$ are of bidegree
$(1,0)$, $F_2$ and $F'_2$ are of bidegree,
and all the fibers are tangent to $E$.
Such configurations depend on {\bf two} parameters. 

We put these in the prescribed determinantal form. First write
$$F_1=\{y_1=0 \}, F'_1=\{z_1=0\}, F_2 = \{y_2=0\}, F'_2=\{z_2=0\}$$
and set 
$$E=\{g(y_1,z_1;y_2,z_2)=0\}, \quad
g \in \bC[y_1,z_1;y_2,z_2]_{(2,2)},$$
with $g$ chosen such that $E$ has the tangencies specified above.
We set
$$ A = \left( \begin{matrix} y_1z_1 & 0  & 0 & 0 \\
			      0     & y_1z_2 & 0 & 0 \\
			      0     &   0    & y_2z_1 & 0 \\
		 	      0     &   0    &   0    & y_2z_2g 
		\end{matrix} \right).$$
Let $\cC_0 \ra \bP^1 \times \bP^1$ denote the quadric surface bundle
associated with this quadratic form.

Pirutka's technique \cite[Th.~3.17]{PirSurvey} shows that $\cC_0$
has unramified cohomology, arising from the pull-back of the class
of $\Br(\bC(\bP^1 \times \bP^1))$ ramified along 
$$F_1 \cup F'_1 \cup F_2 \cup F'_2$$
with higher-order ramification at the four points of intersection.
While $\cC_0$ is singular, it admits a universally
$\CH_0$-trivial resolution
of singularities following the procedure in \cite[\S 5]{HPT16}.
Indeed, the configuration of singularities here is \'etale-locally
equivalent at each point to a stratum of the configuration 
considered in \cite{HPT16}. There we had a cycle of three rational curves,
each with multiplicity two and simply tangent 
to the smooth (multiplicity-one) component
of the degeneracy locus at a point;
here we have a cycle of four rational curves with the same
tangency to $E$.

Now suppose
$$\cC \ra \bP^1 \times \bP^1$$
is branched over a {\em very general}
$(6,6)$ curve. An application of \cite{voisin} or \cite{ct-pirutka}, as in \cite{HPT16}, shows that 
for such  $\cC$ the corresponding fourfold lacks an integral decomposition
of the diagonal and thus fails to be stably rational.

\section{The real case}
\label{sect:real}

\subsection{Normal forms for pencils of quadrics}

Let $X\subset \bP^n$ be a smooth complete intersection of two quadrics
over $\bR$.

Write $X=\{Q_0=Q_1=0\}$ with associated pencil 
$$P=\{ s_0Q_0+s_1Q_1 \} \subset \bP^n \times \bP^1$$
and
binary form of degree $n+1$
$$F(s_0,s_1)=\det(s_0Q_0+s_1Q_1)\in \bR[s_0,s_1].$$
Note that
$F(s_0,s_1)$ has no multiple roots because $X$ is smooth.  We 
write a normal form under linear changes of coordinates,
following \cite[Th.~2]{Thompson}, expressed as an orthogonal
sum of matrix blocks:
\begin{itemize}
\item{associated with nonreal roots $a+bi$ of $F(1,\rho)$ the block
\begin{equation} \label{2block}
\left( \begin{matrix} b & a-\rho  \\  
			a-\rho & -b \end{matrix}\right);
\end{equation}
}
\item{associated with real roots $a$ of $F(1,\rho)$ the block
$$\left( \begin{matrix}\pm(a-\rho)  
			\end{matrix}\right).$$}
\end{itemize}
Provided that $s_0\nmid F(s_0,s_1)$, i.e., $Q_1$ is nondegenerate,
these are the only blocks that may arise.

\subsection{Isotopy classification}
We follow \cite[\S 1]{Krasnov}.

Consider the degree-two covering
$S^1 \ra \bP^1$ obtained by realizing
$$S^1=\{(s_0,s_1): s_0^2+s_1^2=1 \}.$$ 
Let 
$$\widetilde{P}:=P\times_{\bP^1}S^1 \ra P$$
denote the associated covering. The advantage of passing to $S^1$ is
that the equation $s_0Q_0+s_1Q_1$ lifts to a well-defined family of quadratic
forms over $S^1$. (Over $\bP^1$ the forms are defined up to sign.) 

The {\em positive index function} $I^+: S^1 \ra \bZ$ is defined
as the number of positive eigenvalues of the associated form
$s_0Q_0+s_1Q_1$.
It satisfies the following:
\begin{itemize}
\item{$I^+$ is piecewise constant with $2k\le 2n+2$ jumps of
height $\pm 1$;}
\item{$I^+(-s_0,-s_1)=n+1-I^+(s_0,s_1)$ provided $(s_0,s_1)$
is not one of the points of discontinuty.}
\end{itemize}
We are using the fact that $X$ is smooth which means
$F(s_0,s_1)$ has no multiple roots so a quadric
drops rank at $(s_0,s_1) \in S^1$ by at most one.

We extract a combinatorial invariant: A point of discontinuity 
for $I^+$ is {\em positive} if $I^+$ increases as we cross,
moving counter-clockwise. Each positive point of discontinuity
matches with its antipodal negative point of discontinuity. 
Observe that
\begin{itemize}
\item{$k=0$ only when $n$ is odd and $F(s_0,s_1)$ has no real nontrivial
roots, i.e., our pencil is a sum of $\frac{n+1}{2}$ two-dimensional real
blocks (\ref{2block});}
\item{when $k\neq 0$ we have $k\in \{1,\ldots,n+1\}$ and $k\equiv n+1 \pmod{2}$.}
\end{itemize}
This reflects the fact that the number of real roots of a polynomial $p(s) \in \bR[s]$  has
the same parity as $\deg(p)$.

Decompose
$$k=k_1+\cdots + k_{2s+1}, k_i \in \bN,$$
where each $k_i$ represents that number of consecutive positive
points of discontinuity, with the indices increasing
as we move counter-clockwise around $S^1$. The length of the decomposition
is odd because the sign of the quadratic form is reversed under the
antipodal involution of $S^1$. We impose an equivalence relation,
identifying decompositions related by
cyclic permutations or reversing all the terms.

\begin{theo} \cite{Krasnov}
Isotopy classes of smooth complete intersections of two quadrics in $\bP^n$ 
correspond to equivalence classes of odd decompositions
$$k_1+\cdots+k_{2s+1} = k \leq n+1$$
where $k$ is a non-negative integer with parity equal to $n+1$
(allowing $k=0$ when $n$ is odd).
\end{theo}
The isotopy class corresponding to the trivial decomposition of $k=n+1$
is the case where $X(\bR)=\emptyset$.
\begin{exam}
\begin{enumerate}
\item[$(n=2)$]{There are three isotopy classes corresponding to 
$$(1),(3),(1,1,1).$$}
\item[$(n=3)$]{There are four isotopy classes corresponding to
$$(0),(2),(2,1,1),(4).$$}
\item[$(n=4)$]{There are seven isotopy classes corresponding to
$$(1),(3),(1,1,1),(5),(3,1,1),(2,2,1),(1,1,1,1,1).$$}
\item[$(n=5)$]{There are nine isotopy classes corresponding to
$$(0),(2),(4),(2,1,1),(6),(4,1,1),(3,2,1),(2,2,2),(2,1,1,1,1).$$}
\end{enumerate}
\end{exam}

\subsection{Existence of maximal linear subspaces}
Let $X\subset \bP^n$ be a smooth complete intersection of two quadrics and
write 
$$\dim(X)=n-1= 2m+1 \text{ or }2m,$$
depending on the parity of $n$. There exist linear subspaces in $X_{\bC}$ of
dimension $\le m$. The Lefschetz hyperplane theorem shows that larger
dimensional subspaces are not possible.
\begin{theo} \cite[\S 2]{Krasnov}
The variety $X$ contains a real linear subspace of maximal dimension $m$ provided
that 
$$m+1 \le I^+(s_0,s_1) \le m+3$$
for each $(s_0,s_1) \in S^1.$
\end{theo}
\begin{exam}
For threefolds $X\subset \bP^5$ we have the variety of lines $F_1(X)$ admits
real points for the decompositions
$$(0),(2),(2,1,1),(2,2,2),(2,1,1,1,1).$$
These correspond to sequences of nondegenerate signatures on $S^1$:
\begin{align*}
(3,&3)  \\
(2,4) (3,3) & (4,2) (3,3) \\
(2,4) (3,3) (4,2) (3,3) & (4,2) (3,3) (2,4) (3,3) \\
(2,4) (3,3) (4,2) (3,3) (2,4) (3,3) & (4,2) (3,3) (2,4) (3,3) (4,2) (3,3)  \\
(2,4) (3,3) (4,2) (3,3) (4,2) (3,3) & (4,2) (3,3) (2,4) (3,3) (2,4) (3,3) 
\end{align*}
Here we omit the points of discontinuity.
\end{exam}

In particular, decompositions $(4)$, $(4,1,1)$ and $(3,2,1)$ admit real points but no real lines.

\subsection{Topological types in the threefold case}
See \cite[Th.~5.4]{Krasnov} for topological types of complete intersections
of two quadrics $X\subset \bP^5$ over $\bR$. We list the types that admit real
points but not real lines:
\begin{itemize}
\item{$(4)$: $X({\bR})$ is diffeomorphic to the sphere $S^3$;}
\item{$(4,1,1)$: $X({\bR})$ is diffeomorphic to the disjoint
union $S^3 \sqcup S^3$;}
\item{$(3,2,1)$: $X({\bR})$ is diffeomorphic to the product of
spheres $S^1 \times S^2$.}
\end{itemize}
In particular, $X({\bR})$ is disconnected in case $(4,1,1)$ thus irrational over $\bR$.
It is not {\em a priori} clear whether cases $(4)$ or $(3,2,1)$ are rational;
we explore this in Theorem~\ref{theo:IJcrit}. Certainly
there is no topological obstruction to realizing $S^3$ with a rational
threefold
$$\{x_1^2+x_2^2+x_3^2+x_4^2=x_0^2 \} \subset \bP^4.$$
And there is no topological obstruction for $(3,2,1)$ -- case
$(2)$ yields examples
that are rational but diffeomorphic to $S^1 \times S^2$. 
See \cite{KollarNash} for more discussion.

\subsection{A refinement of the intermediate Jacobian criterion}

Let $Y$ be a smooth projective geometrically rational threefold over $\bR$.
If $Y$ is rational over $\bR$ then its intermediate Jacobian is isomorphic
to the Jacobian of a smooth projective (not necessarily irreducible) 
curve $D$ over $\bR$ \cite[Cor.~2.8]{BenWit}
$$\IntJ(Y) \simeq \Jac(D).$$

Fix a family of codimension-two algebraic cycles
$$
\cZ \subset Y \times B
$$
flat over the base $B$. Assume that:
\begin{itemize}
\item{
given $b_0\in B(\bC)$, the morphism 
$$
\begin{array}{rcl}
B_{\bC} & \ra & \IntJ(Y_{\bC}) \\
b & \mapsto & [\cZ_b - \cZ_{b_0}]
\end{array}
$$
is an isomorphism;}
\item{
the Albanese $\Alb(B) \simeq \Jac(D)$ over $\bR$.}
\end{itemize}
Under the first assumption, $B$ carries the structure of a principal 
homogeneous space over an abelian variety, which is isomorphic to
$\Jac(D)$ by the second assumption.

\begin{prop} \label{prop:PHS}
Retain the assumptions above. Assume that
\begin{itemize}
\item
$Y$ is rational over $\bR$;
\item
$\IntJ(Y)$ admits no factors as a principally polarized abelian variety
that are elliptic or the product of two complex conjugate elliptic curves.
\end{itemize}
Then there exist principally polarized
factors $J_i$ of $\IntJ(Y)$,
isoomorphisms $\eta_i:\Jac(D_i) \ra J_i $, and degrees $d_{i}$ 
such that
$$[B]=\sum_{i} {\eta_{i}}_*[\CH^1(D_i)_{d_{i}}].$$
\end{prop}
Elliptic factors make the bookkeeping more complex:
There exist nonisomorphic genus-one curves with isomorphic
Jacobians. By `complex conjugate elliptic curves' we mean that 
complex conjugation interchanges the two factors/components.

\begin{lemm}
Each irreducible factor of $\IntJ(Y)$ is elliptic, a product of
two complex conjugate elliptic curves, or of the form $J_i \simeq \Jac(D_i)$,
where $D_i$ is an irreducible smooth projective 
curve over $\bR$ of arithmetic genus at least two. In the last case,
$D_i$ is uniquely determined by its Jacobian.
\end{lemm}
This follows from 
the Torelli Theorem over nonclosed fields \cite{SerreLauter}.

\begin{proof}
The birational map $\bP^3 \dashrightarrow Y$ admits a factorization as 
iterated blow-ups
and blow-downs along smooth centers \cite{AKMW}:
$$\xymatrix{  Z_0 \ar[dr]  \ar[dd] &  &  \ar[dl] Z_1 & \cdots & Z_{n-1}  \ar[dr]& & Z_n \ar[dl]\ar[dd]  \\
                & Z_{0,1} &   \cdots  & \cdots & \cdots  & Z_{n-1,n} &  \\
            \bP^3 \ar@{-->}[rrrrrr]   &   & & & & &  Y 
}.$$
For notational simplicity we will assume that $n=0$ and write
$Z=Z_0=Z_n$. There is no harm
in doing this as the general case follows by iterating the argument below.

We apply the blow-up formula for a smooth curve $A$ 
in a smooth threefold $W$ \cite[{\S}6.7]{Fulton}
$$\CH^2(\Bl_A(W)) \simeq \CH^1(A) \oplus \CH^2(W).$$
This is compatible with algebraic families: A family of 
codimension-two cycles on $\Bl_A(W)$ induces a morphism from the
base of the family to $\CH^1(A)$.

Thus the intermediate Jacobian of $Z$ in our factorization
is a direct product of the 
Jacobians of the irreducible curves $A_1,\ldots,A_N$ that were blown up 
at various stages and
$$\CH^2(Z) = \bZ \oplus (\oplus_{i=1}^N \CH^1(A_i)).$$
We organize the centers based on which survive in $Y$.
An essential center is a 
connected curve of positive genus whose Jacobian contributes as a factor of
$\IntJ(Y)$.
There may be inessential centers, e.g., positive genus curves
which are blown-up and then blown-down at a subsequent step in the 
factorization. However, these fail to contribute to the Chow group of $Y$.

Decompose our principally polarized abelian variety
$$
\Jac(D) \simeq \IntJ(Y) = \prod_{j=1}^r \Jac(D_j)^{n_j} \times \Jac(E)$$
where the $D_j$ are distinct and of genus at least two,
and $E$ includes the
genus one components.
Under our assumptions, the last factor vanishes.
A principal homogeneous space over such a product is a product of
principal homogenous spaces over the factors.
Consider the corresponding factors in $\IntJ(Z)$ 
$$\IntJ(Z) \simeq \prod_{j=1}^r \Jac(D_j)^{N_j} \times J$$
where $J$ corresponds to the genus-one factors and the higher-genus
factors not among $\{D_1,\ldots,D_r\}.$

The blowup $Z\ra Y$ implies that each $\Jac(D_j)^{n_j}$ sits in 
$\Jac(D_j)^{N_j}$ with projector
$$\Pi_j: \Jac(D_j)^{N_j} \ra \Jac(D_j)^{n_j}.$$
The associated contribution to $\CH^2(Y)$ is given by applying
$\Pi_j$ to $\CH^1(D_j)^{N_j}$, regarded as a sum of
compatible principal homogeneous spaces over $\Jac(D_j)^{N_j}$. 
Interpret $\Pi_j$ as a matrix with entries endomorphisms of
$\Jac(D_j)$. Since this respects principal polarizations, it
takes the shape of a projection onto $n_j$ of the factors,
up to isomorphisms of those factors. Reindex with $i=1,\ldots,n_1+\cdots+n_r$
so that each irreducible factor $\Jac(D_i) \subset \IntJ(Y)$ gets
its own index; the $D_i$ need not be distinct.
Then the summand of $\CH^2(Y)$ 
associated with $D_i$ is obtained by applying these isomorphisms
to cocycles of the form $\CH^1(D_i)_{d_{i}}$ for suitable
degrees $d_{i}$.

Our family of cycles over $B$ gives a morphism of principal homogeneous
spaces
$$B \ra \prod {\eta_{i}}_*(\CH^1(D_i)_{d_{i}})$$
that becomes an isomorphism over $\bC$.
Hence it is an isomorphism over $\bR$.
\end{proof}

\subsection{Application to complete intersections of quadrics}
We return to assuming that $X\subset \bP^5$ is a smooth complete intersection
of two quadrics over $\bR$. We focus on the case where rationality
remains open, i.e., $X(\bR) \neq \emptyset$ but $X$ does not admit a 
real line. Note that $X$ automatically admits a conic over $\bR$,
as the pencil of quadric hypersurfaces containing it admits members
of signature $(3,3)$ that contain isotropic planes.  
Recall from Section~\ref{subsect:QP} that the space of such conics is an \'etale $\bP^3$-bundle over the genus
two curve $C$ associated with the pencil, which thus admits real points.

We apply Proposition~\ref{prop:PHS} to $B=F_1(X)$ and
$\cZ$ the universal family of lines on $X$.
The work of Wang \cite{wang} (see also \cite[Th.~4.8]{ReidThesis})
shows that $F_1(X)$ is a principal
homogeneous space over $\Jac(C)$, satisfying
$$2[F_1(X)]=[\CH^1(C)=\Pic^1(C)].$$
Since $C$ is smooth of genus two, $\Jac(C)$ is simple as a principally
polarized abelian surface -- it cannot be a product of real
or complex conjugate elliptic factors, which are associated with
nodal stable curves of genus two.
Since $C(\bR)\neq \emptyset$ we have $2[F_1(X)]=0$. 
However, then we would have $F_1(X) \simeq \eta_*\CH^1(D_i)_d$ for some $d$
and some endomorphism $\eta:\Jac(D_i) \ra \Jac(D_i)$; here $D_i$ is an
essential center of $\bP^3 \dashrightarrow X$.
Again, the Torelli Theorem \cite{SerreLauter} guarantees that
$D_i\simeq C$. We derive a contradiction as 
$$F_1(X)(\bR) =  \emptyset \quad C(\bR) \neq \emptyset.$$

We summarize this as follows:
\begin{theo} \label{theo:IJcrit}
Let $X \subset \bP^5$ be a smooth complete intersection of two quadrics
over $\bR$ not containing a real line. Then $X$ is not rational over $\bR$.
\end{theo}

\begin{rema}
Colliot-Th\'el\`ene points out that this fails for 
complete intersections of two quadrics $X\subset \bP^4$ over $\bR$. 
For instance, let $X$ be the blowup of a quadric surface that
contains no real lines along two pairs of complex conjugate points. 
\end{rema}


\appendix
\selectlanguage{french}
\section*{Appendice, par J.-L. Colliot-Th\'el\`ene}
\def\thesection{A}

Dans toute cette note, $k$ d\'esigne un corps de caract\'eristique diff\'erente de 2.

\subsection{Quadriques avec une sous-vari\'et\'e de degr\'e impair}

 On a le lemme bien connu  de T. A. Springer (voir \cite[Chap. VII, Thm. 2.3]{lam}) :

\begin{CTlem}\label{springer}  
Si une forme quadratique sur un corps $k$  poss\`ede un z\'ero non trivial sur une
extension impaire du corps de base, alors elle  poss\`ede  un z\'ero non trivial sur $k$. 
\end{CTlem}

La proposition suivante est aussi bien connue \cite[Prop. 68.1]{EKM}.

\begin{CTprop}\label{split} Soit $k$ un corps.
Soit $Q \subset \P^n$, $n \geq 2$,  une quadrique lisse
d\'eploy\'ee, c'est-\`a-dire d\'efinie par une forme quadratique
d\'eploy\'ee.

(i)  La dimension maximale d'un sous-espace 
lin\'eaire de $Q$ est   $[(n-1)/2]$.

(ii) Consid\'erons l'application image directe $CH_{r}(Q) \to CH_{r}(\P^n)=\Z$ 
entre les groupes de Chow de cycles de dimension $r$, avec $0 \leq r \leq n-1$.
Pour $r\leq [(n-1)/2]$, cette application est surjective. Pour $r>[(n-1)/2]$, son image est $2.\Z$:
toute sous-vari\'et\'e int\`egre de $Q$ de dimension $r>[(n-1)/2]$ est de degr\'e pair. 
\end{CTprop}

Il serait surprenant que l'\'enonc\'e suivant n'ait pas \'et\'e d\'ej\`a \'etabli.

\begin{CTtheo}\label{quadriques}
Soit $k$ un corps. Soit $Q \subset \P^n_{k}$, $n \geq 1$, une quadrique lisse.
S'il existe une sous-$k$-vari\'et\'e g\'eom\'etriquement
 int\`egre $W \subset Q \subset \P^n_{k}$ 
 de dimension $r$ et de degr\'e impair dans $\P^n_{k}$, alors $Q$ contient un espace lin\'eaire  $\P^r_{k}$.
\end{CTtheo}
\begin{proof}
Supposons la quadrique donn\'ee par l'annulation d'une forme quadratique $q(x_{0}, \dots, x_{n})$.
Si la forme $q$ est de rang pair et   hyperbolique, alors $n+1=2d$ et $q=0$ contient un espace lin\'eaire
$\P^d_{k}$.
Si $q$ est de rang $n+1$, avec  $n=2d$ et s'\'ecrit comme somme orthogonale d'une forme quadratique hyperbolique de rang $2d$ et d'une forme de rang 1, alors $q=0$ contient un espace lin\'eaire $\P^{d-1}_{k}$.
Dans ces deux cas, d'apr\`es la proposition \ref{split}, la d\'emonstration est achev\'ee.  

Toute $k$-vari\'et\'e  $W \subset  \P^n_{k}$ de degr\'e impair contient des points ferm\'es $P$ de degr\'e
$[k(P):k]$ impair. 
Si de plus   $W$ est g\'eom\'etriquement int\`egre, alors ses points ferm\'es de degr\'e impair sont denses pour la topologie de Zariski de $W$. 
D'apr\`es le lemme \ref{springer},
l'hypoth\`ese exclut donc que la forme quadratique $q$ soit anisotrope.

On peut donc supposer que $q(x_{0}, \dots, x_{n})$ 
s'\'ecrit sous la forme
$$q(x_{0}, \dots, x_{n}) = x_{0}x_{1}+ \dots + x_{2s}x_{2s+1} + g(x_{2s+2}, \dots, x_{n})$$
 avec $s \geq 0$ et
avec $g$ une forme quadratique anisotrope en au moins 2 variables.
 Consid\'erons l'application rationnelle de $\P^n_{k}$ vers $\P^{n -2s-2}$
 envoyant $(x_{0}, \dots,x_{n})$ sur $(x_{2s+2}, \dots, x_{n})$.
 Cette application est d\'efinie hors du ferm\'e d\'efini par
 $$(x_{sr+2}, \dots, x_{n})=(0,\dots, 0).$$
 Sa restriction \`a $W \subset Q $ est donc d\'efinie
 hors du ferm\'e $F \subset W$  d\'efini par ces m\^emes \'equations, donc en dehors du ferm\'e de $W$
 d\'efini par $$(x_{2s+2}, \dots, x_{n})=(0,\dots, 0)$$ et $$ x_{0}x_{1}+ \dots + x_{2s}x_{2s+1}=0.$$
 
 Si $F \neq W$, on a alors une
  application rationnelle de $W$ dans la quadrique anisotrope
 de $\P^{n -2s-2}$ d\'efinie par $ g(x_{2s+2}, \dots, x_{n})=0$.
 Comme les points ferm\'es de degr\'e impair sont Zariski denses dans $W$, et
 que la quadrique anisotrope ci-dessus ne poss\`ede pas de point ferm\'e de degr\'e impair
 par le lemme \ref{springer}, ceci est impossible.
 
On a donc $F =W$. La vari\'et\'e $W$ de dimension $r$ est contenue
 dans le ferm\'e de $\P^n_{k}$  d\'efini   par $(x_{2s+2}, \dots, x_{n})=(0,\dots, 0)$,
 qui est un sous-espace projectif $\P^{2s+1}_{k}$,  et dans la quadrique de cet espace projectif d\'efinie par
$x_{0}x_{1}+ \dots + x_{2s}x_{2s+1}=0$.  Comme $W$ est de degr\'e impair, 
d'apr\`es la proposition \ref{split},
ceci force $r\leq s$.
Ainsi $q=0$ contient un espace lin\'eaire $\P^r_{k}$.
\end{proof}

\subsection{Intersection de deux quadriques avec une sous-vari\'et\'e de degr\'e impair}

Le th\'eor\`eme suivant est d\^{u} \`a M. Amer \cite{amer}.
Le cas $d=1$ fut \'etabli  ind\'ependamment par A. Brumer.
 Le th\'eor\`eme g\'en\'eral est  \'etabli de nouveau, en toute caract\'eristique,
dans un tapuscrit de D. Leep \cite{leep}.

\begin{CTtheo}\label{theoamer}
Soient $f$ et $g$ deux formes quadratiques en $n+1$ variables
sur le corps $k$.  La forme quadratique $f+tg$ s'annule sur un sous-espace
lin\'eaire de  dimension $d$ de $k(t)^{n+1}$ si et seulement si 
$f=g=0$ s'annule sur un espace lin\'eaire de dimension $d$ de $k^{n+1}$. 
\end{CTtheo}

Le cas $n=5,  r=1$ du th\'eor\`eme suivant est \'etabli, par une autre m\'ethode, 
dans \cite[Thm. 14]{HT}.

\begin{CTtheo}\label{deuxquadriques}
Soit $X \subset \P^n_{k}$, $n \geq 3$, une intersection compl\`ete  lisse de deux quadriques. 

(i) S'il existe une sous-$k$-vari\'et\'e g\'eom\'etriquement int\`egre $W \subset X \subset \P^n_{k}$ 
 de dimension $r$ et de degr\'e impair dans $\P^n_{k}$, alors $X$ contient un espace lin\'eaire  $\P^r_{k}$.

(ii) Si de plus $r \geq 1$, alors $X$ est  
  $k$-birationnelle \`a $ \P^{n-2}_{k}$.
   \end{CTtheo}
\begin{proof} 
   Soit $X \subset \P^n_{k}$ d\'efinie par l'annulation  de deux  formes quadratiques $f=g=0$.  
  La quadrique lisse $Q$ sur le corps $K=k(t)$ d\'efinie par $f+tg=0$
  contient la sous-$K$-vari\'et\'e  g\'eom\'etriquement int\`egre $W\times_{k}K$,
  qui est de degr\'e impair et de dimension $r$. D'apr\`es le th\'eor\`eme \ref{quadriques}, elle contient un
  espace lin\'eaire $\P^r_{K}$. D'apr\`es le th\'eor\`eme \ref{theoamer}, la $K$-vari\'et\'e $X$
  contient un espace lin\'eaire $\P^r_{k}$. Ceci \'etablit (i), et (ii) en r\'esulte d'apr\`es
 \cite[Prop. 2.2]{CTSSDI}.
\end{proof}

  \subsection{Intersection de deux quadriques qui contiennent une paire rationnelle de  droites gauches}

On r\'epond ici n\'egativement \`a la question du \S  7 de \cite{HT},
et on donne simultan\'ement des exemples un peu plus simples que ceux du Corollaire 27
du paragraphe 9 de \cite{HT}.

Commen\c cons par des exemples sur le corps des r\'eels, variantes
de \cite[\S 2, p. 128]{CTS} et \cite[\S 1, Prop. 1.3]{CTm}.

\begin{CTprop}\label{reel}
Soient $n \geq 5$ un entier et  $X \subset \P^n_{\R}$, 
 une intersection compl\`ete lisse de deux quadriques donn\'ee par
un syst\`eme d'\'equations homog\`enes :
$$ f(x_{2}, \dots, x_{n}) - x_{0}x_{1}= 0 = g(x_{2}, \dots, x_{n})  -(x_{0}-x_{1})(x_{0}-2x_{1}),$$
avec $f(x_{2}, \dots, x_{n})$ et $g(x_{2}, \dots, x_{n})$ deux formes quadratiques \`a coefficients r\'eels d\'efinies positives.

Alors :

(i) L'espace topologique $X(\R)$ a deux composantes connexes.

(ii) La $\R$-vari\'et\'e $X$ n'est pas stablement rationnelle, ni m\^eme r\'etrac\-tilement rationnelle.

(iii) La $\C$-vari\'et\'e $X_{\C}$ contient un espace lin\'eaire $\P^{m}_{\C}$ avec $m=[(n-3)/2]$
qui ne rencontre pas son congugu\'e complexe.
\end{CTprop}
\begin{proof}
Il n'y a pas de point de $X(\R)$ avec $(x_{0},x_{1})=(0,0)$. On dispose donc de l'application continue
$X(\R) \to \P^1({\R})$ d\'efinie par $(x_{0},x_{1})$. Son image est la r\'eunion des intervalles $[0,1]$ et $[2,\infty]$.
L'espace $X(\R)$ a donc au moins deux composantes connexes, et c'est le maximum possible pour
une intersection lisse de deux quadriques. Pour les cons\'equences de la non connexit\'e de $X(\R)$
sur la non rationalit\'e d'une $\R$-vari\'et\'e projective et lisse $X/\R$, je renvoie aux r\'ef\'erences
donn\'ees dans \cite[Th\'eor\`eme 1.1]{CTm}.
Pour l'\'enonc\'e (iii), il suffit d'observer que la section $Y$ de $X$ par $x_{0}+x_{1}=0$
est une intersection   de deux quadriques dans $ \P^{n-1}_{\R}$ qui satisfait $Y(\R)=\emptyset$.
On sait que toute intersection de deux quadriques dans $ \P^{n-1}_{\C}$ contient un espace lin\'eaire
de dimension $m=[(n-3)/2]$; ceci r\'esulte par exemple  de la combinaison du th\'eor\`eme \ref{theoamer} et du  th\'eor\`eme de Tsen.
Comme on a $Y(\R)=\emptyset$, un tel sous-espace lin\'eaire de $Y_{\C}$ ne saurait rencontrer son conjugu\'e dans $Y_{\C}$.
\end{proof}

La proposition suivante utilise la m\'ethode de sp\'ecialisation sur une vari\'et\'e
pas trop singuli\`ere poss\'edant des invariants non ramifi\'es non triviaux \cite{ct-pirutka}.

\begin{CTprop} Soit $p\neq 2$ un nombre premier.
 Soit $\F$ un corps fini de caract\'eristique $p$ assez gros.
Sur tout corps $K$ avec $\F(x) \subset K\subset \F((x))$,   sur tout
corps $K$ avec $\C(x)(y) \subset K \subset \C((x))((y))$, 
sur tout corps de nombres $K$, et sur tout corps
$p$-adique $K$,   il existe
une intersection lisse de deux quadriques $X \subset \P^5_{K}$
qui contient un $K$-point, qui poss\`ede une paire de droites 
gauches d\'efinies sur une extension quadratique de $K$,
et qui n'est pas r\'etractilement rationnelle, et en particulier n'est
pas stablement $K$-rationnelle.
\end{CTprop}
\begin{proof}

On utilise  les exemples donn\'es avec Coray et Sansuc dans \cite{CTCS}; voir aussi 
la liste d'exemples  de  \cite[\S 15]{CTSSDII}.

Soit $k$ un corps de caract\'eristique diff\'erente de 2
et $a \in k$ non carr\'e. Soient $(x,y,z,t,u,v)$
des coordonn\'ees homog\`enes de $\P^5_{k}$.

Soit $\alpha=\sqrt{a}$.
Soit $X \subset \P^5_{k}$
 d\'efinie par le syst\`eme
$$q_{1} = x^2- a y^2 -  uv = 0$$
$$q_{2} = z^2- a t^2 -   (u-cv)(u-dv)=0,$$
avec $uv(u-cv)(u-dv)$  sans facteur multiple.

Elle contient le point rationnel lisse $M$ de coordonn\'ees
$(1,0,1,0,1,0)$.

La classe de l'alg\`ebre de quaternions $ ((u-cv)/v, a) \in \Br(k(X))$
est non ramifi\'ee sut tout mod\`ele projectif et lisse de $X$. 
Comme $a$ n'est pas un carr\'e dans $k$, cette classe
ne provient pas de $\Br(k)$.
Ces deux \'enonc\'es sont   \'etablis dans \cite[Prop. 6.1 (iii)]{CTCS}.

Le lieu non lisse de $X$ est form\'e des deux points ferm\'es $R$ et $S$
d\'efinis l'un par
$u=v=0=z=t=0$ et donc $x^2-ay^2=0$, l'autre par
 $u=v=x=y=0$ et $z^2-at^2=0$.
 On dispose d'une r\'esolution des singularit\'es $f: \tilde{X} \to X$
qui est un isomorphisme au-dessus du compl\'ementaires de $R$ et $S$
et telle que les fibres $f^{-1}(R)$ et $f^{-1}(S)$ sont des quadriques lisses
de dimension 2
sur le corps $k(\sqrt{a})$ poss\'edant un $k(\sqrt{a})$-point : 
 les $k(\sqrt{a})$-vari\'et\'es $f^{-1}(R)$ et $f^{-1}(S)$ sont donc universellement $CH_{0}$-triviales.

La vari\'et\'e $X$ 
contient deux droites gauches conju\-gu\'ees
$$x-\alpha y= z-\alpha t=u=v=0$$
et 
$$x+\alpha y= z+\alpha t=u=v=0.$$

On d\'eforme maintenant $X$ en une intersection lisse
de deux quadriques
 contenant deux droites conju\-gu\'ees
et contenant le point $M$.  
Il suffit pour cela de  prendre une intersection lisse $f_{1}=f_{2}= 0$
de deux quadriques dans $\P^5_{k}$
contenant les deux droites gauches ci-dessus et contenant le point $M$
(voir \cite[ \S 4 et \S 1]{CTSSDI}; c'est ici que l'on suppose le corps fini assez gros).
On consid\`ere alors  l'intersection compl\`ete  lisse  de deux quadriques  $X_{\lambda}$
sur le corps $K=k(\lambda)$ donn\'ee par
$$ q_{1} + \lambda f_{1}=0,$$
$$q_{2} + \lambda f_{2} =0.$$

La $K$-vari\'et\'e $X_{\lambda}$ 
poss\`ede un point $K$-rationnel et contient deux droites
conjugu\'ees. 
Le th\'eor\`eme  de sp\'ecialisation sous la forme \cite[Thm. 1.12]{ct-pirutka}
montre que  $X_{\lambda}$ 
 n'est pas $K$-r\'etractilement rationnelle.

On peut prendre pour $k$ tout corps assez gros de caract\'eristique diff\'erente de 2
pour lequel $k\neq k^2$. Par exemple un corps fini  $\F$ de caract\'eristique diff\'erente
de 2 assez gros, ou $k=\C((T))$ ou $k=\C(T)$.

L'argument donne ainsi des exemples de $X \subset \P^5_{K}$   non r\'etractilement rationnels sur 
$K=\C((T_{1}))((T_{2}))$, sur $K=\C(T_{1},T_{2})$,  sur $K=\F((T))$, sur $K=\F(T)$.
On peut aussi faire une d\'eformation en in\'egale caract\'eristique et faire des exemples
sur tout corps $p$-adique  (de corps r\'esiduel assez gros et non dyadique), et de l\`a sur tout corps de nombres.
\end{proof}

On trouvera un exemple analogue pour les hypersurfaces cubiques de $\P^4_{\Q_{p}}$
dans \cite[Thm. 1.21]{ct-pirutka}.

\subsection{Intersection de deux quadriques non rationnelles en dimension quelconque sur des corps non ordonnables}

 On a d\'ej\`a donn\'e de tels exemples sur les r\'eels (Prop. \ref{reel}).
 On peut  donner  des exemples sur des corps non ordonnables.
 L'argument donn\'e ici a d\'ej\`a \'et\'e d\'evelopp\'e
 dans  \cite[Thm. 4.1]{CTm} pour les hypersurfaces cubiques diagonales.
 On renvoie \`a \cite{CTm} pour plus de d\'etails sur les outils employ\'es (cohomologie galoisienne,
 cohomologie non ramifi\'ee).

Soit $k$ un corps de caract\'eristique diff\'erente de 2, contenant $a\in k^{*}\setminus  k^{*2}$.
Soit $K_{n}=k(s_{1}, \dots,s_{n})$ le corps des fonctions rationnelles en $n \geq 0$ variables.
Soient   $b_{1}, \dots, b_{n} \in k$ et
$X_{n} \subset \P^{n+4}_{K_{n}}$ l'intersection compl\`ete  de deux quadriques donn\'ee par
le syst\`eme d'\'equations :
$$\phi=x^2-ay^2 -uv +  \sum_{i=1}^{n} s_{i} y_{i}^2=0$$
$$ \psi= 2(x^2-az^2) - (u+v)(2u-v) +   \sum_{i=1}^{n} b_{i }s_{i} y_{i}^2=0$$
en les variables homog\`enes $(x,y,z,u,v,y_{1}, \dots, y_{n})$.
 La $K_{n}$-vari\'et\'e $X_{n}$ poss\`ede le $K_{n}$-point  $(x,y,z,u,v,y_{1}, \dots, y_{n})=(1,0,0,1,1,0,\dots,0)$. On supposera $X$ lisse. C'est le cas si le polyn\^{o}me homog\`ene $det(\lambda\varphi+\mu \psi)$ est s\'eparable.
 Si $k$ est infini ou fini avec assez d'\'el\'ements, il existe des \'el\'ements
 $b_{1}, \dots, b_{n} \in k$ qui satisfont ces conditions.

\begin{CTprop}
Soit $K_{n}(X_{n})$ le corps des fonctions de la $K_{n}$-vari\'et\'e $X_{n}$.
Le cup-produit $$\alpha_{n}:=((u+v)/v, a, s_{1}, \dots, s_{n}) \in H^{n+2}(K_{n}(X_{n}),\Z/2)$$ des classes 
de $((u+v/v), a, s_{1}, \dots, s_{n}) $ dans  
$$K_{n}(X)^*/K_{n}(X_{n})^{*2}=H^1(K_{n}(X_{n}),\Z/2)$$
est  non  ramifi\'e sur la $K_{n}$-vari\'et\'e $X_{n}$,  et n'appartient pas \`a l'image de $H^{n+2}(K_{n},\Z/2)$. Ainsi la $K_{n}$-vari\'et\'e $X_{n}$
n'est pas $CH_{0}$-triviale et en particulier n'est pas r\'etrac\-tilement rationnelle.
\end{CTprop}
\begin{proof}
Pour $n=0$, la classe de quaternions $((u+v)/v, a)$ est un \'el\'ement non constant de la 2-torsion du groupe
de Brauer de la surface $X_{0} \subset \P^4_{k}$, voir \cite[\S 4]{BSD}. Cela d\'efinit donc une classe dans $H^2_{nr}(k(X_{0})/k,\Z/2)$
qui ne vient pas de $H^2(k,\Z/2)$. On notera que cela exclut la pr\'esence d'un couple de droites gauches
conjugu\'ees sur $X_{0}$.

Soit $n \geq 1$. Supposons l'\'enonc\'e d\'emontr\'e pour $n-1$.  Sur la $K_{n}$-vari\'et\'e lisse $X_{n}$,  
la classe $\alpha_{n}$ est non ramifi\'ee en dehors des diviseurs 
d\'efinis par $v=0$ et par $u+v=0$. Le diviseur $\Delta$ d\'efini par $v=0$ est int\`egre, donn\'e par le syst\`eme
$$x^2-ay^2   -  \sum_{i=1}^{n} s_{i} y_{i}^2=0,$$
$$ 2(x^2-az^2) - 2u^2 -   \sum_{i=1}^{n} b_{i }s_{i} y_{i}^2=0.$$
Le r\'esidu de $\alpha_{n}$ en $\Delta$ est le cup-produit
$$\beta_{n}:= (a, s_{1}, \dots, s_{n}) \in H^{n+1}(K_{n}(\Delta),\Z/2).$$
L'identit\'e $x^2-ay^2   -  \sum_{i=1}^{n} s_{i} y_{i}^2=0$  sur $\Delta$ implique que $\beta_{n}$ est nul.s
L'argument sur le diviseur int\`egre d\'efini par $u+v=0$ est identique. La classe $\alpha_{n}$
est donc non ramifi\'ee sur la $K_{n}$-vari\'et\'e $X_{n}$.

On consid\`ere par ailleurs le mod\`ele propre sur $K_{n-1}[s_{n}]$ d\'efini par le m\^{e}me syst\`eme d'\'equations
que $X_{n}$. La fibre au-dessus de $s_{n}=0$ n'est autre que le c\^one sur la $K_{n-1}$-vari\'et\'e $X_{n-1}$
d\'efinie par le syst\`eme d'\'equations :
$$x^2-ay^2   -  \sum_{i=1}^{n-1} s_{i} y_{i}^2=0,$$
$$ 2(x^2-az^2) - 2u^2 -   \sum_{i=1}^{n-1} b_{i }s_{i} y_{i}^2=0.$$
Le r\'esidu de $\alpha_{n}$ au point g\'en\'erique de cette vari\'et\'e est
la classe 
$$\alpha_{n-1}=((u+v)/v, a, s_{1}, \dots, s_{n-1}),$$ 
qui par hypoth\`ese de r\'ecurrence n'est pas
dans l'image de $H^{n+1}(K_{n-1}, \Z/2)$. Comme la fibre $s_{n}=0$ a multiplicit\'e 1,
la comparaison des r\'esidus en $s_{n}=0$ montre que
 $\alpha_{n}$ n'est pas dans l'image de $H^{n+2}(K_{n}, \Z/2)$.
\end{proof}

L'argument ci-dessus peut s'adapter en in\'egale caract\'eristique
et donne  sur tout corps $p$-adique $K$   avec $p\neq 2$
des exemples d'intersection lisse de deux quadriques $X \subset \P^5_{K}$ 
non r\'etractilement rationnelle.

\selectlanguage{english}

\bibliographystyle{alpha}
\bibliography{2quadwCT}
\end{document}